\documentclass[11pt,reqno,a4paper]{amsart}

\usepackage{graphicx} 
\usepackage[USenglish]{babel}
\usepackage[utf8]{inputenc}
\usepackage{amsfonts,mathrsfs}
\usepackage{amssymb, mathabx}
\usepackage{amsthm}
\usepackage{url}
\usepackage[backgroundcolor=orange!30!white,linecolor=orange!80!white,textsize=footnotesize]{todonotes}
\usepackage{comment}
\usepackage{tikz,tikz-cd, tikz-3dplot}
\usepackage{mathtools}
\usepackage{upgreek}
\usepackage{enumitem}
\usepackage{aliascnt}
\usepackage{dsfont}
\usepackage{epic}
\usepackage{indentfirst}
\usepackage{framed}
\usepackage{subfig}
\usepackage{caption}
\usepackage{subcaption}
\usepackage[a4paper,top=3.3cm,bottom=3cm,left=3cm,right=3cm,bindingoffset=5mm]{geometry}
\usepackage{xcolor}
\usepackage[colorlinks = true,
            linkcolor = blue,
            urlcolor  = blue,
            citecolor = red,
            anchorcolor = blue,
pdftitle={},
pdfauthor={},
linktocpage=true
]{hyperref}
\addto\extrasUSenglish{}
\usetikzlibrary{babel, calc, fadings, intersections}
\usetikzlibrary{decorations.pathreplacing,angles,quotes}

\DeclareMathOperator{\vol}{vol}

\DeclareMathOperator{\Pic}{Pic}
\DeclareMathOperator{\Prin}{Prin}

\DeclareMathOperator{\val}{val}
\DeclareMathOperator{\Sym}{Sym}
\newcommand{\rBN}{\mathrm{BN}}
\newcommand{\ind}{\mathrm{ind}}

\newtheorem{theorem}{Theorem}[section]
\newaliascnt{lemma}{theorem}\aliascntresetthe{lemma}
\newtheorem{thm}[theorem]{Theorem}
\newaliascnt{prop}{theorem}\newtheorem{prop}[prop]{Proposition}\aliascntresetthe{prop}
\newaliascnt{cor}{theorem}\aliascntresetthe{cor}
\newaliascnt{question}{theorem}\aliascntresetthe{question}
\newaliascnt{defin}{theorem}\newtheorem{defin}[defin]{Definition}\aliascntresetthe{defin}

\newaliascnt{conj}{theorem}\aliascntresetthe{conj} 
\theoremstyle{remark}
\newaliascnt{rem}{theorem}\newtheorem{rem}[rem]{Remark}\aliascntresetthe{rem} 
\newaliascnt{example}{theorem}\newtheorem{example}[example]{Example}\aliascntresetthe{example}

\newcommand{\BC}{{\mathbb{C}}}

\newcommand{\BG}{{\mathbb{G}}}
\newcommand{\BN}{{\mathbb{N}}}
\newcommand{\BP}{{\mathbb{P}}}

\newcommand{\BR}{{\mathbb{R}}}

\newcommand{\BZ}{{\mathbb{Z}}}
\newcommand{\BT}{{\mathbb{T}}}
\newcommand{\Z}{{\mathbb{Z}}}
\newcommand{\R}{{\mathbb{R}}}

\newcommand{\mcE}{{\mathcal{E}}}

\newcommand{\mcN}{{\mathcal{N}}}
\newcommand{\mcV}{{\mathcal{V}}}

\newcommand{\mcF}{{\mathcal{F}}}

\DeclareMathOperator{\trop}{trop}

\DeclareMathOperator{\Supp}{Supp}

\DeclareMathOperator{\Div}{Div}
\DeclareMathOperator{\ddiv}{div}

\DeclareMathOperator{\PL}{PL}

\theoremstyle{plain}
\newtheorem{introthm}{Theorem}

\begin{document}
	
\setlength{\parindent}{2ex}
	
\title{Asymptotic behavior of tropical rank functions}

\author[A. Botero]{Ana Maria Botero}
\address{Ana Maria Botero, Universit\"{a}t Bielefeld, Universit\"{a}tsstra{\ss}e 25, D-33615 Bielefeld, Germany}
\email{abotero@math.uni-bielefeld.de}

\author[A. K\"{u}ronya]{Alex K\"{u}ronya}
\address{Alex K\"{u}ronya, Institut f\"{u}r Mathematik, Goethe-Universit\"{a}t Frankfurt, Robert-Mayer-Str. 6-10., D-60325 Frankfurt am Main, Germany}
\email{kuronya@math.uni-frankfurt.de}

\author[E. Vital]{Eduardo Vital}
\address{Eduardo Vital, Universit\"{a}t Bielefeld,  Universit\"{a}tsstra{\ss}e 25, D-33615 Bielefeld, Germany}
\email{evital@math.uni-bielefeld.de}

\date{\today.}

\begin{abstract}
 
    We show that the asymptotic behavior of the two main competing notions of rank of a linear series on a tropical curve is governed by asymptotic invariants, closely paralleling the theory of volumes in algebraic geometry. We introduce and study tropical notions of volume associated to both divisors and tropical modules. We prove optimal asymptotic results for each case. In addition, we show that the tropical volume is compatible with the tropicalization of curves.
\end{abstract}
\maketitle
	
\section{Introduction}\label{Section: Introduction}

Divisor theory on graphs and tropical curves has developed into a powerful combinatorial counterpart of classical algebraic geometry. Since the foundational work of M.~Baker and S.~Norine in \cite{Baker_Norine_2007}, linear series on graphs have provided a discrete analog of the theory of divisors on algebraic curves, leading to tropical versions of the Riemann--Roch theorem and Brill--Noether theory \cite{trop-brill-noether}. Our purpose here is to initiate a study of the asymptotic point of view.

\medskip

A central invariant of a divisor on a curve is the \emph{rank} of the associated linear series. In the tropical setting, several notions of rank have been proposed. The \emph{Baker--Norine rank}, defined originally for graphs and extended to metric graphs, captures the combinatorial structure of chip-firing, and has been fundamental for the development of Brill--Noether theory. Other notions of rank on a metric graph, such as the \emph{independence rank}, introduced by D.~Jensen and S.~Payne in \cite{Jensen_Payne_2022}, arise from considering a notion of tropical linear independence. Although these notions agree (up to $\pm 1$) in many cases, they do not coincide in general, and understanding their relationship remains an important open problem in tropical geometry; see e.\,g.\,\cite[\S\ 9]{trop_linear_series_matroids_25}. 

\medskip

In classical algebraic geometry, an effective way to study linear series is through \emph{asymptotic invariants}. Given a divisor $D$ on a projective variety $X$, one studies the a\-symptotic growth of the spaces of sections $H^0\big((X, \mathcal{O}_X(mD)\big)$, as $m \to \infty$. The resulting invariant---called the \emph{volume} of $D$---has remarkable formal properties: it depends only on the numerical class of $D$, it is homogeneous of degree the dimension of $X$, and it is log-concave on the cone of big divisors. These properties play an important role in modern birational geometry and the study of linear series; see for instance \cite{Lazarsfeld_2004_I} or \cite{Cutkosky_Zariski_dec}, as well as in the theory of asymptotic cohomology developed in \cite{alex-asymp-cohom}. In situations with a strong combinatorial flavour, such as the case of toric varieties, the volume admits an especially precise description in terms of convex geometry \cite{Hering_Kueronya_Payne}.

\medskip

Motivated by these developments, we study asymptotic invariants of linear series on multigraphs and tropical curves, focusing on the growth rate of tropical ranks. More concretely, given a divisor $D$ on a multigraph, we consider the asymptotic behavior of the Baker--Norine rank of the linear series $R(mD)$ associated to $mD$,
as $m \to \infty$. In the case where $D$ is a divisor on a tropical curve, we further consider the asymptotic behavior of the independence rank of $D$. As we shall see, the asymptotic behavior of the Baker--Norine and of the independence rank of $D$ are equal; see \autoref{thm: deg=asymptotic ind_dim trop}. Now, given any tropical submodule $\Sigma \subseteq R(D)$ we further define a multiplication map $\Sigma \times \dotsm \times \Sigma \to R(mD)$ and study the asymptotics of both the Baker--Norine and independence rank as $m \to \infty$. Asymptotic behavior of rank functions leads to natural notions of \emph{tropical volumes}, see Definitions \ref{definition: Volume of a divisor}, \ref{def: Vol_C(D)} and \ref{def: trop+ind_volumes}

\medskip

The main result of this paper is that the asymptotic behavior of a divisor $D$ on a tropical curve $\Gamma$ is independent of the choice of the notion of rank one chooses. In particular, the associated tropical volume is determined by $\Gamma$ and $D$ only, and satisfies an asymptotic Riemann--Roch formula. In contrast, the Baker--Norine and the independence volume of a tropical submodule $\Sigma \subseteq R(D)$ do not agree in general, in particular, the volume of $D$ and the volume of its associated tropical module $R(D)$ do not agree in general; see \autoref{rem: vol_D_and_t_mod}. We provide a sequence of sharp inequalities linking both.  
 
\begin{introthm}[\autoref{thm: deg=asymptotic ind_dim trop}]\label{th:introA}

    Let $\Gamma$ be a tropical curve, equivalently a metric multigraph, and $D \in \Div(\Gamma)$ a divisor. Let $\Sigma \subseteq R(D)$ be a tropical submodule.
 \begin{enumerate}
     \item[(i)] For $D$, we have  
    \[
        \vol_{\ind, \Gamma}^{\BT}(D) = \vol_{\rBN,\Gamma}^{\BT}(D) = \max\{\deg(D),0\}.
    \]
    \item[(ii)] 
    For $\Sigma$, we have  the following sequence of inequalities 
    \begin{equation*}
        \max\{\deg(D),0\}\ge \vol_{\ind, \Gamma}^{\BT}(\Sigma)\ge \vol_{\rBN, \Gamma}^{\BT}(\Sigma) \ge \max\{r_{\rBN}(\Sigma),0\}.
    \end{equation*}
Moreover, all inequalities are sharp. 
    \end{enumerate}

\end{introthm} 
    
\begin{introthm}[\autoref{thm:asymptotic-rr}]
    For $D$ we obtain an asymptotic Riemann--Roch formula
    \[
        \chi(\ell D)=\deg(D)\ell + o(1),
    \]
    where $\ell \in \mathbb{N}$ and $\chi(\ell D)$ is the Euler characteristic of the divisor $\ell D$.
\end{introthm}
One has a priori two Euler charactertistics $\chi_{\ind}$ and $\chi_{\rBN}$ (corresponding to independence and Baker--Norine ranks) that do not agree in general. In fact, \autoref{exa: R-R does not holt for r_ind} shows that Riemann--Roch is not satisfied for the independence rank in general. Even so, the above asymptotic result holds for both. In \autoref{rem:tropi-implies-combi} we comment on the relation between the combinatorial and the metric multigraph case.

\medskip
  
Let us note here that the combinatorial statement in \autoref{lemma: lower_bound} is not immediately implied by the metric multigraph statement in \autoref{lemma: lower_bound_trop_curves}. Namely, given a multigraph $G$, we set the length of each edge $e\in E(G)$ to be $1$, to define $\Gamma$, the metric multigraph associated with $G$. When $G$ is not only a singleton, then both $\Div(\Gamma)$ and $\Prin(\Gamma)$ are strictly larger than $\Div(G)$ and $\Prin(G)$, respectively. In \cite[pg.\ 618]{Baker_2008} it was conjectured that the BN-rank in $G$ and $\Gamma$ would agree. It is easily verified that the conjecture is not true in full generality (take the example of a multigraph $G$ that consists of just one loop). Nevertheless, it was shown in \cite[Thm.\ 1.3]{Hladk_2013} that the conjecture holds for loopless multigraphs.

\medskip
   
If $\Gamma$ is a multigraph, then \autoref{prop: deg=asymptotic dim} provides an asymptotic Riemann--Roch formula. We point out that there exists no Riemann--Roch theorem for multigraphs (i.e.\ allowing loops), in particular, the Riemann--Roch theorem of Baker--Norine is not valid for multigraphs. Indeed, consider the multigraph $G$ which is a loop with one vertex.  There have been various proposals to treat loops for the purpose of Riemann--Roch; e.g.\ by considering weights on the vertices; see \cite{Amini_Caporaso_2013}. \autoref{prop: deg=asymptotic dim} shows that an asymptotic version always holds (even without weights).

\medskip

It follows directly from \autoref{th:introA} that the resulting tropical volume $\vol_{\Gamma}^{\BT}(\cdot )$ shares several fundamental properties of the algebraic setting: it is homogeneous of degree one,  and it is log-concave on the space of divisor classes.   This provides some initial evidence that the tropical volume is the correct asymptotic invariant of linear series in this setting. 
Finally, we show that the tropical volume function is well-behaved with respect to the tropicalization map; see \autoref{prop:specialization} and \autoref{eq:vol-tro}. 

\medskip

The paper is organized as follows. In \autoref{sec: Ranks of linear systems on graphs}, we study the asymptotic behavior of the Baker--Norine rank for divisors on multigraphs, introduce the notion of tropical volume, and prove \autoref{prop: deg=asymptotic dim}, showing that the asymptotic behavior of the $\rBN$-rank depends only on the degree; we also establish an asymptotic Riemann--Roch theorem for multigraphs in \autoref{thm:asymptotic-rr_graph}. 
In \autoref{sec: Ranks of linear systems on tropical curves}, we extend this analysis to tropical curves, where the main result, \autoref{prop: deg=asymptotic dim trop. curve}, shows that tropical volume admits the same simple description as in the combinatorial case. In \autoref{Sec: JP rank}, we introduce the independence rank of $\BT$-submodules (i.e.\ tropical submodules) of $R(D)$ and prove that it has the same asymptotic behavior as the Baker–Norine rank, \autoref{thm: deg=asymptotic ind_dim trop}, leading to a unified notion of tropical volume and an asymptotic Riemann--Roch theorem, \autoref{thm:asymptotic-rr}. Finally, in \autoref{Sec: tropicalization_Lin_Ser}, we relate these tropical results to classical algebraic geometry via tropicalization, showing that tropical volume is compatible with this process; see \autoref{prop:specialization}.

\medskip

\noindent \textbf{Acknowledgements.} Our collaboration started at the workshop ``Geometry of Semirings'' at the Universitat Aut\`{o}noma de Barcelona. We would like to thank the organizers, Joaquim Ro\'{e} and  Marc Masdeu, for the pleasant atmosphere and the excellent working conditions. We also thank Matthew Dupraz for useful comments on a previous version of this article.

A.~M.~Botero and E.~Vital were funded by the Deutsche Forschungsgemeinschaft (DFG, German Research Foundation) – Project-ID 491392403 – TRR 358. A.~K\"{u}ronya acknowledges partial support from  Deut\-sche Forschungsgemeinschaft (DFG) through the Collaborative Research
Centre TRR 326 ``Geometry and Arithmetic of Uniformized Structures" (project number 444845124), and from the ANR-DFG project ``Positivity on K-trivial varieties'' (project number ANR-23-CE40-0026).

\section{Asymptotics of the Baker--Norine rank of linear series on multigraphs}\label{sec: Ranks of linear systems on graphs} 
 
Here, we study the asymptotic behavior of the Baker–Norine rank for divisors on multigraphs i.e.\ we allow loops and multiple edges. We prove an asymptotic Riemann--Roch theorem for multigraphs. We also introduce the notion of tropical volume via the growth of ranks of multiples of a divisor,  establish bounds for the rank,  and prove that the asymptotic behavior depends only on the degree. This culminates in an explicit formula for the tropical volume in the combinatorial setting.

\medskip

Let $S$ be a finite set, and $n\in\BN$. We define $\binom{S}{n}\coloneqq \{ S'\subseteq S \mid \#S'=n\}$.
A \emph{multigraph} $G$ is a triple $(V,E,\varepsilon)$ where $V$ is the set of \emph{vertices}, $E$ is the set of \emph{edges}, and $\varepsilon: E\to \binom{V}{1} \cup \binom{V}{2}$.
In words, a multigraph is a graph that allows loops and multiple edges. We say that a multigraph $G=(V, E, \varepsilon)$ is finite if both $V$ and $E$ are finite. Here, we consider only finite connected multigraphs. 
Given a vertex $v$ of $G$, its \emph{valence} $\val(v)$ is the number of edges adjacent to it. If $e\in E(G)$ is a loop, we consider that $\val\!\big(\varepsilon(e)\big)=2$; we refer the reader to \cite{Diestel_2025} for more details concerning multigraphs. 

\subsection*{Divisors on multigraphs} Let $G$ be a multigraph. We define  $\Div(G)\coloneqq \bigoplus_{v\in V(G)}\Z v$, the group of divisors on $G$, the free abelian group generated by $V(G)$. An element $D\in\Div(G)$ is called a \emph{divisor} on $G$, and we write 
\[
    D=\sum_{v\in V(G)} d_v v, \quad \text{ for } \quad d_v \in \Z\,.
\]
We use the definitions and notation of the theory of divisors on graphs from \cite{Baker_Norine_2007}, with the difference that here we allow loops.
 
If $G$ is a loop at the vertex $v$, then chip-firing on $D$ at $v$ returns the same divisor $D$. Moreover, $D\sim D'$ if and only if there exists a finite sequence of chip-firing that takes $D$ and returns $D'$; see \cite[Lem.\ 4.3]{Baker_Norine_2007}. 

As chip-firing is a commutative operation, for $S \subseteq V(G)$ we fire once from each vertex of $S$ and call it a set-firing on $D$ at $S$. 

\subsection*{Linear series on multigraphs} Let $G$ be a multigraph, and $D$ a divisor on $\Div(G)$. Following \cite[\S\ 1.6]{Baker_Norine_2007} we define

\begin{defin}\label{def: complete linear series for multigraphs}
    The \emph{complete linear series} of a divisor $D\in \Div(G)$, is:
    \[
        |D|\coloneqq \{ D'\in \Div_{+}(G)\; | \;  D'\sim D \}.
    \]
    The  \emph{Baker--Norine rank} (or BN-rank) of $|D|$ is defined as
    \begin{equation*}\label{def: rank D}
        r(D)\coloneqq 
        \begin{cases}
            -1,  \text{ if } |D|=\varnothing, \\
            \max\{d \; | \;  \forall E \in \Div_{+}^{d}(G), \; |D-E|\neq\varnothing \}, \text{ otherwise}.
        \end{cases}
    \end{equation*}
\end{defin}

\begin{rem}\label{rem: rank_negative_degree}
    If $D$ has negative degree, then $r(D)=-1$. 
    In this case, it follows that $-1\leq r(\ell D) \leq 0$ for all $\ell \in \BN$, with $r(\ell D)=0$ only if $\ell=0$. 
    Moreover, if $\deg(D)\ge-1$, then $\deg(D)\ge r(D)$. Namely, suppose $\deg(D)< r(D)$, then for all $E\in\Div_{+}^{r(D)}(G)$, we have $|D-E|\ne\varnothing$. Thus, there exists $E'\in |D-E|$, and as linear equivalence preserves degree, it follows that 
    $
        \deg(E')= \deg(D)-r(D)<0,
    $
a contradiction. 
\end{rem}
\begin{defin}\label{definition: Volume of a divisor}
    The \emph{tropical volume} of a divisor $D\in\Div(G)$ is 
    \[
        \displaystyle \vol_{G}^{\BT}(D)\coloneqq \limsup_{\ell\to\infty}\frac{r(\ell D)}{\ell}.
    \]
    We say that $D$ is \emph{big} if $\vol_G^{\BT}(D)>0$. 
\end{defin}
As the space $R(D) \coloneqq \left\{f \in \mathcal{M}(G) \; | \; D + \operatorname{div}(f) \geq 0 \right\}$ is a $\BT$-module (see \autoref{sec: Ranks of linear systems on tropical curves}), and its rank is defined to be $r(D)$, that is, the BN-rank of the linear series $|D|$, we have the justification for the adjective ``tropical" in the definition of the volume.

\begin{rem}\label{rem:BN-inequalities}
 By \cite[Lem.\ 2.1]{Baker_Norine_2007}, $r(D+D') \ge r(D)+r(D')$, whenever both $r(D)$ and $r(D')$ are nonnegative. 
In particular, if $r(D)\ge 0$, then $r(\ell D) \ge \ell r(D)$ for any $\ell\in\BN$. Therefore, for any $\ell,\ell'\in \BN$ it follows that 
\[
    r(\ell D+\ell' D')\ge r(\ell D) + r(\ell'D')\ge \ell r(D) + \ell' r(D').
\]
Thus, $\deg(\ell D) \ge r(\ell D) \ge \ell r(D)$.
Dividing by $\ell$ and taking upper limits with respect to  $\ell \to \infty$, we obtain
\begin{equation}\label{eq: r_asymptotic}
    \deg(D) \ge \vol_G^{\BT}(D) \ge r( D). 
\end{equation}   
\end{rem}

Next, fix a divisor $D\in\Div(G)$, and consider the map that sends $\ell \in \BZ_{>0}$ to ${r(\ell D)}/{\ell}$.
If $r(D)\ge0$, then Inequalities \eqref{eq: r_asymptotic} allow us to represent the asymptotic behavior of ${r(\ell D)}/{\ell}$ as shown in \autoref{fig:asymptotic_a_D}.
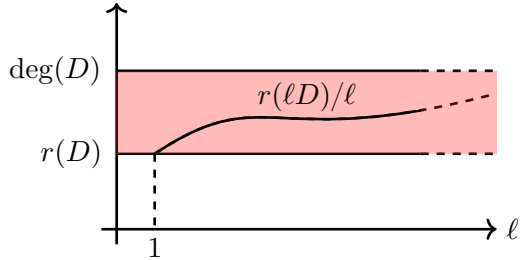
\begin{figure}[ht]
    \centering
    \begin{tikzpicture}[line width = 1pt]   
        \draw[->] (0,-.2) -- (0,3) node[yshift=0pt, xshift=-10pt] {};
        \draw[->] (-.2,0) -- (5,0) node[xshift=6pt] {$\ell$};
        \draw[name path=deg] (4,2.1) -- (0,2.1) node[xshift=-23pt] {$\deg(D)$}; 
        \draw[dashed, name path=deg_d] (4,2.1) -- (5,2.1) node[xshift=-23pt] {};
        \draw[line width = .9pt,name path=rank] (4,1) -- (0,1) node[xshift=-17pt] {$r(D)$}; 
        \draw[dashed, name path=rank_d] (4,1) -- (5,1) node[xshift=-17pt] {};
        \draw[dashed, name path=1] (.5,1) -- (.5,0) node[yshift=-7pt] {1};
        \path[dashed, draw, name path=a_d] (.5,1) .. controls (2,2) and (2,1) .. (5,1.8);
        \fill[color=red!90,opacity=.30] 
            (0,1) -- (5,1) -- (5,2.1) -- (0,2.1);

        \path[name path = 4] (4,1) -- (4,2.3);
        \path[name intersections={of=4 and a_d,by={a}}];
        \begin{scope}
            \clip (.5,1) rectangle (a);
            \path[-, draw, name path=a_d_d] (.5,1) .. controls (2,2) and (2,1) .. (5,1.8);
        \end{scope}
        \node[] at (2.5,1.75) {${r(\ell D)}/{\ell}$};
    \end{tikzpicture}
    \caption{Asymptotic behavior of ${r(\ell D)}/{\ell}$.}
    \label{fig:asymptotic_a_D}
\end{figure}

A special case of Inequalities \eqref{eq: r_asymptotic} is when $\deg(D)=r(D)$, which implies $\vol_G^{\BT}(D)=\deg(D)$. We illustrate this case in the following Example.
\begin{example}\label{exa: deg(D)=vol(D)}

Consider the multigraph $G$ in \autoref{fig: graph_Gamma_tree}.
    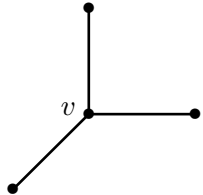
\begin{figure}[ht]
        \centering
        \begin{tikzpicture}[line width = 1pt]
            \node[] (v)      at   (-.27,.07)      {$v$} ;
            \node[inner sep=-5pt] (v0)     at   (0,0)        {$\bullet$} ;
            \node[inner sep=-5pt] (v1)     at   (1.4,0)      {$\bullet$} ;
            \node[inner sep=-5pt] (v2)     at   (0,1.4)      {$\bullet$} ;
            \node[inner sep=-5pt] (v3)     at   (-1,-1)      {$\bullet$} ;

            \draw[-, inner sep=-5pt] (v0) to (v1);
            \path[-, draw] (v0) -- (v2);
            \path[-, draw] (v0) -- (v3);
        \end{tikzpicture}
        \caption{A multigraph $G$, where $r(D)=\deg(D)$.}
        \label{fig: graph_Gamma_tree}
    \end{figure}
    We choose and fix a vertex, say $v$, in $G$ to identify $\Pic(G)=\Z v$.
    Thus, for any divisor $D$ with $\deg(D)\ge-1$, it follows that $r(D)=\deg(D)$. Indeed, given $D\in\Div^d(G)$, it follows that $D\sim d v$. By Inequalities \eqref{eq: r_asymptotic}, it follows that $\vol_G^{\BT}(D) = \deg(D)$.
\end{example}

In the equality $\deg(D)=\vol_G^{\BT}(D)$ in \autoref{exa: deg(D)=vol(D)}, the graph $G$ is special, it is a tree. In \autoref{prop: deg=asymptotic dim} we prove that the equality $\deg(D)=\vol_G^{\BT}(D)$ holds in general.
First, we have the following Proposition:

\begin{prop}\label{lemma: lower_bound}
    There exists $C_{G}\in\BN$ (only depending on $G$) such that $r(D) \ge \deg(D) - C_G$ for all $D\in\Div(G)$. 
\end{prop}

Before giving a proof, we set up some notation. For any $D=\sum_{v\in V(G)}d_vv\in\Div(G)$, we define $\mathcal{F}(D) \coloneqq \{ v\in V(G)\mid d_v\ge \val(v) \}$  and $\mathcal{N}(D) \coloneqq \{ v\in V(G)\mid d_v < \val(v) \}$. Note that $\mcF(D)$ and $\mcN(D)$ form a partition of $V(G)$.
Given $D$, there exist unique subdivisors $D^+$ and $D^-$ of $D$, with $D^+$ and $-D^-$ in $\Div_+(G)$ such that $D=D^+ + D^-$. 

We now prove \autoref{lemma: lower_bound}.

\begin{proof}
    Let $C_G\coloneqq  2\big(\#E(G)+1\big)\big(\#V(G)+1\big)$, and note that if $\deg(D)< C_G$, then $r(D)\ge-1\ge \deg(D)-C_G$. 

    Assume $\deg(D)\ge C_G$ and let $D'\in\Div_{+}(G)$ with $\deg(D')=\deg(D)-C_G$. We prove that $|D-D'|\neq\varnothing$. Write $E\coloneqq D-D'=\sum_{v\in V(G)}e_vv$, and observe that $\deg(E)=C_G$.
    Since $\deg(E)=C_G> 2 \#E(G)\#V(G)$,
    it follows that $\mcF(E)$ is not empty. Using  chip-firing, fire from all $v\in \mcF(E)$ to obtain a new divisor $E_1\sim E$.     
    Starting with $E_1$, fire from all $v\in \mcF(E_1)$ to obtain $E_2\sim E_1$. Repeating this, we have a sequence of linear equivalent divisors 
    \[
        E\coloneqq E_0 \sim E_1\sim E_2\sim\cdots.
    \]
    For each $\ell\in\BN$, write 
    \begin{equation}\label{eq:coeff}
        E_\ell=\sum_{v\in V(G)}e_{v,\ell}v\ ,
    \end{equation}
    and note that:
    \begin{enumerate}[label = (F\arabic*)]
        \item \label{item: FP1} $0\le e_{v,\ell+1} \le e_{v,\ell}$ for all $v\in \mcF(E_\ell)$;
        \item \label{item: FP2} $e_{v,\ell+1} \ge e_{v,\ell}$ for all $v\in \mcN(E_\ell)$ and $e_{v,\ell+1} > e_{v,\ell} $ for some $v\in \mcN(E_\ell)$.
    \end{enumerate}

    For each $\ell\in\BN$, define $\mcN_\ell\coloneqq \bigcup_{i=0}^{\ell}\mcN(E_{i})$. As $\mcN_\ell\subseteq \mcN_{\ell+1}$ for each $\ell\in \BN$, and $V(G)$ is finite, 
    there exist $N_0\in\BN$ such that $\mcN_\ell=\mcN_{N_0}$ for all $\ell\ge N_0$.
    If $v\in \mcN_{N_0}$, then $e_{v,\ell} < 2\val(v)$ for all $\ell\ge N_0$. 
    Indeed, if $v\in \mcN_{N_0}$, then $v\in \mcN(E_\ell)$ for some $\ell\le N_0$, and  as $e_{v,\ell} < \val(v)$, it follows that $e_{v,\ell+1}\le e_{v,\ell} + \val(v)<2\val(v)$. Furthermore, if for some $\ell'\in\BN$ we have $\val(v)\le e_{v,\ell+\ell'} < 2\val(v)$, then by \ref{item: FP1} it follows that $e_{v,\ell+\ell'+1}\le e_{v,\ell+\ell'}<2\val(v)$.

    The set $\mcN_{N_0}$ is a proper subset of $V(G)$, indeed, if $\mcN_{N_0}=V(G)$, then 
    \[
        \deg(E_{N_0}) = \sum_{v\in \mcN_{N_0}} e_{v,N_{0}} =\sum_{v\in V(G)} e_{v,N_{0}} < \sum_{v\in V(G)} 2\val(v) \le C_G=\deg(E_{N_0})\ .
    \]
    Let $p\in V(G)\smallsetminus\mcN_{N_0}$, and note that $e_{p,\ell}\ge \val(p)$ for all $\ell\in\BN$. If $\deg(E_{N_{0}}^-)=0$, then we are done. If not, then there exists $q\in \mcN_{N_0}$ such that $e_{q,N_0}<0$. As $G$ is connected, there exists a path connecting $p$ to $q$; see \autoref{fig: path_p_q}.
    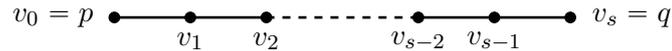
\begin{figure}[ht]
        \centering
        \begin{tikzpicture}
            \draw[line width = 1pt, -] (0,0) -- (2,0) node[] {};
            \draw[line width = 1pt, -, dashed] (2,0) -- (4,0) node[] {};
            \draw[line width = 1pt, -] (4,0) -- (6,0) node[] {};
            \draw[fill] (0,0) circle [radius=2pt] node[xshift=-23pt] {$v_0=p$};
            \draw[fill] (1,0) circle [radius=2pt] node[yshift=-9pt] {$v_1$};
            \draw[fill] (2,0) circle [radius=2pt] node[yshift=-9pt] {$v_2$};
            \draw[fill] (4,0) circle [radius=2pt] node[yshift=-9pt] {$v_{s-2}$};
            \draw[fill] (5,0) circle [radius=2pt] node[yshift=-9pt] {$v_{s-1}$};
            \draw[fill] (6,0) circle [radius=2pt] node[xshift=23pt] {$v_s=q$};
        \end{tikzpicture}
        \caption{A path in $G$, from $p$ to $q$.}
        \label{fig: path_p_q}
    \end{figure}
    
    If $e_{v_1,N_0}<\val(v_1)$, then there exists $\ell_1\le \val(v_1) - e_{v_1,N_0}$ such that after $\ell_1$ firings from $\mcF(E_{N_0})$---which includes firing from $p\in \mcN_{N_0}$---we obtain $e_{v_1,N_0+\ell_1}\ge \val(v_1)$. This means that: 1) $e_{v_1,\ell}\ge 0$ for all $\ell\ge N_0+\ell_1$, and; 2) to increase $e_{v_2,N_0}$ we need at most $\val(v_1) - e_{v_1,N_0}+1$ firings. Thus, after finitely many increases of $e_{v_2,N_0}$, say $\ell'_2$, we have $e_{v_2,N_0+\ell'_2}\ge \val(v_2)$, then there exists $\ell_2\le\ell_2'$ such that $e_{v_3,N_0+\ell_2+1}>e_{v_3,N_0+\ell_2}$. Now, it is not difficult to see that after finitely many firings, say $\ell$ firings 
    \begin{equation}\label{eq: e_l increases}
        e_{q,N_0+\ell}> e_{q,N_0}.
    \end{equation}
    
    If $e_{v,\ell}<0$ for some $\ell$ and $v$, then $e_{v,\ell+1}\ge e_{v,\ell}$ by \ref{item: FP2}. This and the Inequality \eqref{eq: e_l increases} imply that there exists $\ell\in\BN$ such that $\deg(E_{N_0+{\ell}}^-)>\deg(E_{N_0}^-)$. Since $-\infty<\deg(E^-_{N_0})$, after finitely many firings we reach $E_{N}$ such that $\deg(E_N^-)=0$, i.e.\ $0\le E_N\sim E$.
\end{proof}

\begin{rem}
  In algebraic geometry, Fujita's approximation theorem implies that the (classical) volume is actually a limit; see Theorem 11.4.4 and Example 11.4.7 in \cite{Lazarsfeld_2004_II}. In our case, \autoref{lemma: lower_bound} implies that the tropical volume is also a limit, that is
\[
    \vol_G^{\BT}(D)=\lim_{\ell\to\infty}\frac{r(\ell D)}{\ell}.
\]

The constant $C_G$ in \autoref{lemma: lower_bound} is not optimal, but for our purposes we do not need an optimal constant.   
\end{rem}

\begin{example}\label{exa: illustration lemma lower_bound}
    Let $G$ be a multigraph with $4$ vertices and $4$ edges, as in \autoref{fig: exa_divisor_illus_lemma_lower_bound}.
    \begin{figure}[ht]
        \centering
        \begin{tikzpicture}[line width = 1pt, dot/.style={circle, fill, inner sep=1pt, minimum size=2pt}]
            \coordinate (v0)   at    (.5,-.5);
            \coordinate (v1)   at   (.5,.5);
            \coordinate (v2)   at   (-.5,.5);  
            \coordinate (v3)   at   (-.5,-.5); 
            
            \draw[fill] (v0) circle [radius=2pt] node[xshift=7pt, yshift=-7pt] {$v_0$};
            \draw[fill] (v1) circle [radius=2pt] node[xshift=7pt, yshift=7pt] {$v_1$};
            \draw[fill] (v2) circle [radius=2pt] node[xshift=-7pt, yshift=7pt] {$v_2$};
            \draw[fill] (v3) circle [radius=2pt] node[xshift=-7pt, yshift=-7pt] {$v_3$};

            \draw[-] (v0) to (v3);
            \draw[-] (v0) to (v1);
            \draw[-] (v2) to (v1);
            \draw[-] (v2) to (v3);
        \end{tikzpicture}
        \caption{Graph $G$.}
        \label{fig: exa_divisor_illus_lemma_lower_bound}
    \end{figure} 
    Thus $C_G=2(4+1)(4+1) = 50$. Let $D, D'\in\Div(G)$ be divisors given by
    \[
        D=90v_0 - 8v_1 - 8v_2 - 8v_3 \quad \text{ and } \quad
        D' = 10v_0 + 2v_1 + 2v_2 + 2v_3.
    \]
    As in (the proof of) \autoref{lemma: lower_bound}, $\deg(D)\ge C_G$. Now, define 
    \[
        E\coloneqq D-D'=80v_0 - 10v_1 - 10 v_2 - 10v_3.
    \]
    Note that $\mcF(E)=\{v_0\}$ and $\mcN(E)=\{v_1,v_2,v_3\}$. Firing from $\mcF(E)$, repeatedly, $12$ times, we obtain the divisor
    \[
        E_{12}=56v_0 + 2v_1 - 10 v_2 +2v_3.
    \]
    Our new partition of $V(G)$ is given by $\mcF(E_{12})=\{v_0,v_1,v_3\}$ and $\mcN(E_{12})=\{v_2\}$. Firing from $\mcF(E_{12})$ we obtain 
    \[
        E_{13}=56v_0 + v_1 - 8v_2 +v_3.
    \]
    
    At this moment, it is not difficult to realize that after firing repeatedly $8$ more times, we obtain $E_{21}=48v_0 +v_1+v_3$,
    that is, the effective divisor $E_{21}\in |D-D'|$. 
    
    We could use any effective divisor $F$ of fixed degree $\deg(D')$, and we would also conclude that $|D-F|\neq \varnothing$. This shows that $r(D)\ge \deg(D) - C_G.$  Here, with notation as in the proof of \autoref{lemma: lower_bound}, the set $\mcN_{N_0}$ equals to $\{v_1,v_2,v_3\}$, and in particular $v_0\in V(G)\smallsetminus \mcN_{N_0}$.
\end{example}

\begin{rem}\label{rem: D(v) ge val(v)}
    If the degree of $D$ is large enough, say $\deg(D)\ge 2C_G$, then by continuing chip-firing as in the proof of \autoref{lemma: lower_bound}, we eventually obtain an effective divisor $E_n\sim D$ such that $e_{v,n}\ge \val(v)$ for all vertices $v\in V(G)$, where the $e_{v,n}$ are the coefficients appearing in \autoref{eq:coeff}. Moreover, setting $\mcF_{2}(D)\coloneq\{v\in V(G)\mid d_v\ge 2\val(v)\}$ and $\mcN_2(D)\coloneqq V(G)\smallsetminus \mcF_{2}(D)$, we observe  that as $\deg(D)\ge 2C_G > \sum_{v\in V(G)} 4\val(v)$, it follows that $\mcF_2(D)$ is not empty. Thus, mimicking the steps in the proof of \autoref{lemma: lower_bound} we can conclude that, after finitely many firings, we obtain $0\le E_n\sim D$ such that $e_{n,v}\ge \val(v)$ for each vertex $v\in V(G)$.  
\end{rem}

\begin{thm}\label{prop: deg=asymptotic dim}
    Let $D\in\Div(G)$, then $\vol_G^{\BT}(D) = \max\{\deg(D),0\}$.
\end{thm}
\begin{proof}
    If $\deg(D)\leq0$, then by \autoref{rem: rank_negative_degree}  for any $\ell\in \BN$, we have $-1\leq r(\ell D) \leq 0$. Thus, it follows that  $\vol_G^{\BT}(D)=0$.
    \autoref{lemma: lower_bound} implies that $r(\ell D)\ge \ell\deg(D) - C_G$ for  all $\ell$. If $\deg(D)>0$, then by \autoref{rem: rank_negative_degree} $\deg(\ell D)\ge r(\ell D)$ for all $\ell\in\BN$. Thus,
    \[
        \frac{\ell\deg(D)}{\ell} \ge \frac{r(\ell D)}{\ell}\ge \frac{\ell\deg(D)-C_G}{\ell},
    \]
    for all $\ell>0$. Applying $\limsup$ with $ \ell\to\infty$, we finish the proof.
\end{proof}
 
Given $D\in\Div(G)$ we define its \emph{Euler characteristic} as 
\[
    \chi(D)\coloneqq r(D) - r(K_G-D),
\]
where $K_G$ denotes the canonical divisor on $G$.
Since our multigraphs $G$ allow loops, the Euler characteristic $\chi(D)$ does not satisfy the Riemann--Roch theorem in general. The following theorem shows that, asymptotically, it behaves as expected.
\begin{thm}[Asymptotic Riemann--Roch for multigraphs]\label{thm:asymptotic-rr_graph}
    Let $D\in\Div(G)$, then
    \[
        \chi(\ell D)=\deg(D)\ell + o(1).
    \]
\end{thm}
\begin{proof}
    If $\deg(D)\ge0$, then $r(K_G-\ell D)$ is bounded. Thus, by \autoref{prop: deg=asymptotic dim}
    \[
        \limsup_{\ell\to\infty}\frac{\chi(\ell D) - \deg(D)\ell}{\ell}=0.
    \]
    If $\deg(D)<0$, then $r(\ell D)\le 0$ for all $\ell\in\BN$. For a large enough fixed $\ell'\in\BN$ and all $\ell \in \BN$, it follows that $\deg\!\big((K_G-\ell'D)-\ell D\big)\geq -1$ and $r(K_G-\ell'D)\ge0$ by \autoref{lemma: lower_bound}. Thus, by \autoref{rem: rank_negative_degree} and by \cite[Lem.\ 2.1]{Baker_Norine_2007} it follows that
    \begin{equation*}
        \deg\!\big((K_\Gamma - \ell'D) - \ell D\big) \ge r\big((K_\Gamma - \ell'D) - \ell D\big) \ge r(K_\Gamma - \ell'D)+ r(-\ell D).
    \end{equation*}
    The inequalities above and \autoref{prop: deg=asymptotic dim} imply that 
    \[
        \limsup_{\ell\to\infty}\frac{\chi(\ell D) - \deg(D)\ell}{\ell}=0.
    \]
    This finishes the proof.
\end{proof}

\section{Asymptotics of the Baker--Norine rank of linear series on tropical curves}\label{sec: Ranks of linear systems on tropical curves}

In this section, we extend the asymptotic analysis of the Baker–Norine rank to tropical curves, viewed as metric multigraphs. After recalling the relevant definitions of divisors and linear series in this context, we adapt the combinatorial techniques we have seen in \autoref{sec: Ranks of linear systems on graphs} to the metric setting. The main result shows that the tropical volume retains the same simple description as in the discrete case.  Here we use definitions and notations as in \cite{Gathmann_2008} and \cite{Mikhalkin_Zharkov_2008}.

\begin{defin}[Tropical curve, {\cite[Def.\ 3.1]{Mikhalkin_Zharkov_2008}}]\label{def: tropical curve}
    A \emph{tropical curve} is a pair $(C,\mathcal{A})$, where $C$ is a locally finite simplicial complex of dimension $1$, and $\mathcal{A}$ is a complete tropical structure on $C$.  
\end{defin} 
Vertices of $C$ with valence $1$ are called \emph{endpoints}. The \emph{finite part} $C^\circ$ of a tropical curve $C$ is the complement of the endpoints.
 
Compact tropical curves can be seen as metric multigraphs; see \cite[Prop.\ 3.6]{Mikhalkin_Zharkov_2008}. We recall the definition:
\begin{defin}[Metric multigraph{, \cite[Def.\ 1.1]{Gathmann_2008}}]\label{exe: Graphs as tropical curves}   
A \emph{metric multigraph} $\Gamma$ is a topological space obtained from a finite multigraph by assigning to each edge $e$ a length $ \ell(e) \in \mathbb{R}_{>0}\cup\{\infty\}$, and identifying $e$ with the interval $[0,\ell(e)]$. Each edge $e$ of length $\infty$ is identified with $[0,\infty]$, and is such that the end  $\infty$ has valence $1$. We endow $\Gamma$ with the path metric $d_{\Gamma}:\Gamma\times \Gamma \to \BR_{\ge0}\cup\{\infty\}$ induced by these lengths. Let $e$ be an edge of $\Gamma$, we write $d_{e}\coloneqq d_{\Gamma}|_{e}$. 
\end{defin}

A \emph{divisor} $D=\sum d_vv$ on a tropical curve $C$ is an element of the abelian group $\Div(C)\coloneqq \bigoplus_{p\in C}\Z p$. The \emph{support} of $D$ is $\Supp(D)\coloneqq\{v\in C\mid d_v\neq 0\}$. A \emph{rational function} on $C$ is a continuous function $f : C \to \BR \cup \{\pm\infty\}$ such that the restriction of $f$ to any edge of $C$ is a piecewise linear integral affine function with a finite number of pieces. Let $D, D'\in\Div(C)$, we say that $D$ is \emph{equivalent} to $D'$, and write $D\sim D'$ if $D=D'+\ddiv(f)$ for some rational function $f$; see \cite[\S\ 1, 3]{Gathmann_2008}.
 
\begin{defin}\label{def: complete linear series for metric graphs}
    Let $D\in\Div(C)$, its \emph{complete linear series} is defined as 
    \[
        |D|\coloneqq \{ D'\in \Div_{+}(C) \mid D'\sim D\}.
    \]
    Its \emph{Baker--Norine rank}  (or BN-rank) 
    $r(D)$ is defined as: 
    \begin{equation*}
        r(D)\coloneqq 
        \begin{cases}
            -1,  \text{ if } |D|=\varnothing, \\
            \max\!\big\{d \; | \;  |D-E|\neq\varnothing\; \forall E \in \Div_{+}^{d}(G) \big\}, \text{ otherwise}.
        \end{cases}
    \end{equation*}
\end{defin}

\begin{rem}\label{rem: prop_trop_linear_series}
Although in a different setup, the rank of a divisor on a tropical curve is defined exactly as the rank of a divisor on a multigraph; see \autoref{def: complete linear series for multigraphs}. Thus, the same results stated in \autoref{rem: rank_negative_degree} for divisors in multigraphs are true for divisors on tropical curves.
\end{rem}

Let $|D_1|$ and $|D_2|$ be non-empty, and $E_i\in\Div_{+}^{r(D_i)}(C)$ for $i=1,2$. Thus $|D_i-E_i|$ is non-empty for each $i=1,2$, which implies that $|D_1+D_2 - (E_1+E_2)|\neq\varnothing$. Therefore,
\begin{equation}\label{ineq: dim|D|_is_convex}
    r(D_1+ D_2)\ge r(D_1) + r(D_2).
\end{equation}

If both $|D_1|$ and $|D_2|$ are empty, then the Inequality \eqref{ineq: dim|D|_is_convex} holds; however, if only one of them is empty, then the Inequality \eqref{ineq: dim|D|_is_convex} does not hold. Indeed, let $D_1$ be an effective divisor with $r(D_1)>0$ and $D_2$ a divisor with $\deg(D_2)<-\deg(D_1)$. As $\deg(D_1+D_2)<0$, it follows that $r(D_1+D_2)=-1 < r(D_1)+r(D_2)$.

On the tropical curve $C$, each interior point of an edge can be considered as a vertex of valence $2$. Therefore, we have $\#V(C)=\infty$. To distinguish these vertices, we define 
$\mathcal{V}(C)\coloneqq \{v\in V(C) \mid \val(v)\neq 2\}$. This is a finite set. 
In the special case when $C$ is a loop, we fix a point $p\in C$ and put $\mcV(C)=\{p\}$. In the same direction, we represent the set of edges of $C$ by $\mcE(C)$.

Given an edge $e$ of $C$, we write $e^\circ\coloneqq e\smallsetminus \varepsilon(e)$ for its interior points. 
As before, given a divisor $D=\sum_{}d_v v\in\Div(C)$, we write $D=D^+ + D^-$ where both $D^+$ and $-D^-$ are effective subdivisors of $D$. Define  
    \begin{align*}
        \mcF(D)\coloneqq \{p\in \mcV(C)\mid d_p\ge \val(p) \}, \quad \text{ and } \quad
        \mcN(D) \coloneqq \mcV(C)\smallsetminus \mcF(D).
    \end{align*}
    
    For $U\subseteq C$, we set $D|_{U}\coloneqq \sum_{v\in\Supp(D)\cap U}d_v v$.
   
\begin{prop}\label{lemma: lower_bound_trop_curves}
    There exists $c_C\in\BN$ such that $r(D) \ge \deg(D) - c_C$, for all $D$.
\end{prop}    
 
\begin{proof}
    Let $c_C\coloneqq 2\big(\#\mcV(C) + 1\big)\big(\#\mcE(C)+1\big)$. If $\deg(D)<c_C$, then $r(D)\ge-1\ge\deg(D)-c_C$.
    
    Thus, assume $\deg(D)\ge c_C$ and let $F'$ be an arbitrary effective divisor of degree $\deg(D)-c_C$. Let $F\coloneqq D-F'$ and note that $\deg(F)=c_C$. If exists $e\in\mcE(C)$ with $\deg(F^+|_{e^\circ})\ge 2$, let the endpoints of $e$ be $\varepsilon(e)=\{v_1,v_2\}$, note that $e$ can be a loop i.e.\ $v_1=v_2$. Let $d_{C}$ be the path metric from \autoref{exe: Graphs as tropical curves}, and $p_i \in\Supp(F^+|_{e^\circ})$, for $i=1,2$, be such that
    \[
        s_i\coloneqq d_{e}\big(v_i,p_i\big) = \min\!\left\{d(v_i, p) \; | \; p \in \Supp(F^+)\right\}\quad \text{ for } \quad i=1,2,
    \]
    and assume $s_1\leq s_2$. As $s_1$ is minimal, there exists $f_e$, constant in $C\smallsetminus \{e\}$, such that
    \[
        \ddiv(f_e)= v_1 - p_1 - p_2 + \widetilde{p}_2;
    \]
    see \autoref{fig: f on the edge e}.
    \begin{figure}[ht]
        \centering
        \begin{tikzpicture}[line width = 1pt]   
            \draw[red, line width = 1pt, -] (1,1.7) -- (1.7,1) -- (2.6,1) -- (3.3,1.7) -- (4,1.7) node[red, xshift=8pt] {$f_e$};
            \draw[line width = 1pt, -] (1,0) -- (4,0) node[xshift=10pt] {$e$};
            \draw[fill] (1,0) circle [radius=2pt] node[yshift=-9pt] {$v_1$};
            \draw[fill] (4,0) circle [radius=2pt] node[yshift=-9pt] {$v_2$};
            \draw[fill] (1.7,0) circle [radius=2pt] node[yshift=-9pt] {$p_1$};
            \draw[fill] (2.6,0) circle [radius=2pt] node[yshift=-9pt] {$p_2$};
            \draw[fill] (3.3,0) circle [radius=2pt] node[yshift=-9pt] {$\widetilde{p}_2$};

            \draw[dashed, -] (1,0) -- (1,1.7) node[] {};
            \draw[dashed, -] (1.7,0) -- (1.7,1) node[] {};
            \draw[dashed, -] (2.6,0) -- (2.6,1) node[] {};
            \draw[dashed, -] (3.3,0) -- (3.3,1.7) node[] {};
            \draw[dashed, -] (4,0) -- (4,1.7) node[] {};
            \draw[decoration={brace,raise=3pt},decorate]  (2.6,0) -- node[above=3pt] {$s_1$} (3.3,0);
            \draw[decoration={brace,raise=3pt},decorate]  (1,0) -- node[above=3pt] {$s_1$} (1.7,0);
            \node[] at (1.45,1.55) {$1$};
            \node[] at (2.90,1.55) {$1$};
        \end{tikzpicture}
        \caption{The function $f_e$ on $e$, which is constant in $C\smallsetminus \{e\}$.}
        \label{fig: f on the edge e}
    \end{figure}
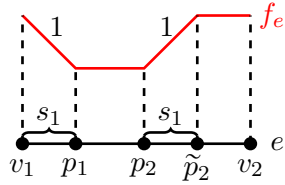
    Note that $F_1\coloneqq F + \ddiv(f_e) \sim F$ has the following properties: 
    \begin{enumerate}[label = (CD\arabic*)]
        \item\label{item: CD1} $\deg(F_1|_{\mcV(C)}) \ge \deg(F|_{\mcV(C)})$, and the inequality is strict if there exists $e_1\in \mcE(C)$ with $\deg(F^+|_{e_1^\circ})\ge 2$;
        \item\label{item: CD3} $\deg(F_1^-) \ge \deg(F^-)$;
        \item\label{item: CD4} $d_{e}(\widetilde{p}_2,v_2)=d_{e}(p_2,v_2)-s_1$.
    \end{enumerate}
    
    As our only condition is $\deg(F^+|_{e^\circ})\ge 2$, we repeat this procedure until we reach $F_n\sim F$ such that $\deg(F_n|_{e^\circ})\le 1$ for all edges $e$. We say that $F_n$ has \emph{degree concentrated} at $\mcV(C)$. 
    Thus, we assume that the degree of $F$ is concentrated at $\mcV(C)$.

    As $\deg(F|_{\mcV(C)})\ge c_{C}-\#\mcE(C)\ge \sum_{v\in\mcV(C)}\val(v)$, it follows that  $\mcF(F)$ is not empty.     
    For each $v\in \mcF(F)$ let $\{\alpha_{1,v},\dots,\alpha_{\val(v),v}\}$ be the set of edges adjacent to $v$. Let $p_{i,v}$ in $\alpha_{i,v}^\circ$, for $i=1,\dots,\val(v)$, be such that it realize the distance 
    \[
        s_v\coloneqq d_{C}\!\left(v, \big(\mcV(C)\cup\Supp(F^-)\big)\smallsetminus\{v\}\!\right).
    \]
    
    Let $f_v$ be a global rational function with slope $1$ from $v$ to $p_{i,v}$  for $i=1,\dots,\val(v)$, which is possible because $s_v$ is minimal, and $f_v$ is constant in the complement of the $s_v$-neighborhood of $v$; see \autoref{fig: rational function f_v}.
    Therefore,  by construction
    \[
        \ddiv(f_v)= -\val(v)v + \sum_{i=1}^{\val(v)} p_{i,v} \quad \text{ for each } \quad  v\in \mcF(F).
    \]
    Thus, $F\sim F_1\coloneqq  F + \sum_{v\in \mcF(F)}\ddiv(f_v)$.
\begin{figure}[h!]
\tdplotsetmaincoords{70}{120}
\begin{tikzpicture}[tdplot_main_coords, scale=2.7, line width=1pt]
    \pgfmathsetmacro{\a}{.8}
    \pgfmathsetmacro{\b}{.3}
  \node[] (0) at (0,0,0)  {$\bullet$};
  \node[] (v) at (0,0,-.15)  {$v$};
  \node[] (psvx) at (-.5,0,0)  {$\bullet$};
  \node[] (svx) at (-.6,0,-.15)  {$p_{1,v}$};
  \node[] (psvy) at (0,.5,0)  {$\bullet$};
  \node[] (svy) at (0,.5,-.15)  {$p_{2,v}$};
  \node[] (psvn) at (.36,-.36,0)  {$\bullet$};
  \node[] (svn) at (.36,-.36,-.15)  {$p_{3,v}$};
  \node[red] (fv) at (.85,-.85,\a)  {$f_v$};

  \draw[-] (0,0,0) -- (0,1,0);
  \draw[-] (0,0,0) -- (.73,-.73,0);
  \draw[-] (0,0,0) -- (-1,0,0);

  \draw[-,red] (0,0,\b) -- (0,.5,\a);
  \draw[-,red] (0,.5,\a) -- (0,1,\a);
  \draw[-,red] (0,0,\b) -- (-.5,0,\a);
  \draw[-,red] (-.5,0,\a) -- (-1,0,\a);
  \draw[-,red] (0,0,\b) -- (.36,-.36,\a);
  \draw[-,red] (.36,-.36,\a) -- (.73,.-.73,\a);

  \draw[-,dashed, line width=1pt] (0,0,\b) -- (0,0,0);
  \draw[-,dashed, line width=1pt] (0,.5,\a) -- (0,.5,0);
  \draw[-,dashed, line width=1pt] (-.5,0,\a) -- (-.5,0,0);
  \draw[-,dashed, line width=1pt] (.36,-.36,\a) -- (.36,-.36,0);
\end{tikzpicture}
\caption{The global rational function $f_v$ in a $s_v$-neighborhood of the vertex $v$, with $\val(v)=3$.}
\label{fig: rational function f_v}
\end{figure}
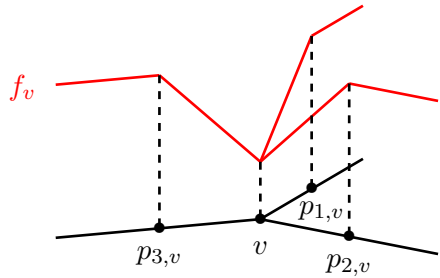

Note that, $\deg(F_1^-)\ge\deg(F^-)$, and with strict inequality if $p_{i,v}\in\Supp(F^-)$ for some $v\in \mcF(F)$ and $i\in\{1,\dots,\val(v)\}$.  If necessary, concentrate the degree of $F_1$ to obtain a divisor $E_1\sim F_1\sim F$. Repeating this, we obtain divisors $E_{\ell}=\sum e_{v,\ell} v$ with $E_{\ell+1}\sim E_{\ell}$, each $E_\ell$ has concentrated degree at $\mcV(C)$, and $-\infty <\deg(E_{\ell}^-)\le \deg(E_{\ell+1}^-)\le 0$ for each $\ell\in\BN$. We set $E_0\coloneqq F$. Furthermore:
\begin{enumerate}[label = (ES\arabic*)]
    \item \label{prop: ES1} $0\le e_{v,\ell+1}\le e_{v,\ell}$ for all $\ell$, and $v\in\mcF(E_\ell)$;
    \item \label{prop: ES2} $e_{v,\ell+1}\ge e_{v,\ell}$ for each $\ell\in\BN$, and $v\in \mcN(E_\ell)$. Thus $\deg(E_{\ell+1}^-)\ge\deg(E^-_{\ell})$.
\end{enumerate}

Define $\mcV_\ell\coloneqq \bigcup_{i=0}^\ell \mcN(E_i)\subseteq\mcV(C)$. As $\mcV_{\ell}\subseteq\mcV_{\ell+1}$ and $\mcV(C)$ is finite, it follows that there exists $N_0\in\BN$ such that $\mcV_{N_0+\ell}=\mcV_{N_0}$ for all $\ell\in\BN$. Now, note that if $v\in \mcV_{N_0}$, then $e_{v,N_0+\ell} < 2\val(v)$ for all $\ell\in\BN$. Indeed, if $v\in \mcV_{N_0}$, then $v\in\mcV_{\ell'}$ for some $\ell'\le N_0$. Thus $e_{v,\ell'}<\val(v)$, and $e_{v,\ell'+1}<2\val(v)$ since after a ``chip-firing", $e_{v,\ell'}$ can increase by at most $\val(v)$. If for some $\ell>\ell'$ we have $\val(v)\le e_{v,\ell}<2\val(v)$, then $0\le e_{v,\ell+1}\le e_{v,\ell}<2\val(v)$ by \ref{prop: ES1}.

If $\deg(E_{N_0}^-)=0$, then $0\le E_{N_0}\sim E$ and the proof is finished. If $\deg(E_{N_0}^-)<0$, we prove that 
\begin{equation}\label{ine: inc_neg_degree}
    \deg(E_{N_0+\ell}^-)>\deg(E_{N_0}^-), 
    \text{ for some } \ell\in\BN.
\end{equation}
As $\deg(E_{N_0}^-)$ is finite, the repeated use of the inequality in \eqref{ine: inc_neg_degree} implies $\deg(E_{N_0+\ell}^-)=0$ for some $\ell\in\BN$.

The set $\mcV_{N_0}$ is a proper subset of $\mcV(C)$. Indeed, if not then for each $\ell\ge N_0$
\begin{align*}
    \deg(E_\ell)  = \hspace{-9pt}\sum_{v\in\Supp(E_{\ell})}\hspace{-13pt} e_{v,\ell}\le \#\mcE(C) + \hspace{-5pt}\sum_{v\in\mcV(C)} \hspace{-7pt}e_{v,\ell} \le \#\mcE(C) + \hspace{-5pt}\sum_{v\in\mcV(C)}\hspace{-7pt} \val(v) 
     <  c_C=\deg(E_\ell),
\end{align*}
which is a contradiction. 

If $\deg(E_{N_0+1}^-)>\deg(E_{N_0}^-)$, then we are done. If not, let $q\in \Supp(E_{N_0}^-)$ and $p\in \mcV(C)\smallsetminus \mcV_{N_0}$ be such that $d(p,q)$ is minimal. 
As $C$ is connected, there exists a path that connects $p$ to $q$ and realizes $d(p,q)$, see \autoref{fig: path_p_q_C},
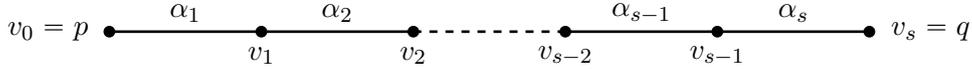
\begin{figure}[ht]
        \centering
        \begin{tikzpicture}
            \node[yshift=7pt] at (0,0) {$\alpha_1$};
            \node[yshift=7pt] at (2,0) {$\alpha_2$};
            \node[yshift=7pt] at (6,0) {$\alpha_{s-1}$};
            \node[yshift=7pt] at (8,0) {$\alpha_{s}$};
            \draw[line width = 1pt, -] (-1,0) -- (3,0) node[] {};
            \draw[line width = 1pt, -, dashed] (3,0) -- (5,0) node[] {};
            \draw[line width = 1pt, -] (5,0) -- (9,0) node[] {};
            \draw[fill] (-1,0) circle [radius=2pt] node[xshift=-23pt] {$v_0=p$};
            \draw[fill] (1,0) circle [radius=2pt] node[yshift=-9pt] {$v_1$};
            \draw[fill] (3,0) circle [radius=2pt] node[yshift=-9pt] {$v_2$};
            \draw[fill] (5,0) circle [radius=2pt] node[yshift=-9pt] {$v_{s-2}$};
            \draw[fill] (7,0) circle [radius=2pt] node[yshift=-9pt] {$v_{s-1}$};
            \draw[fill] (9,0) circle [radius=2pt] node[xshift=23pt] {$v_s=q$};
        \end{tikzpicture}
        \caption{A minimal path in $C$ that connects $p$ to $q$.}
        \label{fig: path_p_q_C}
    \end{figure}
where $v_i\in\mcV(C)$ for $i=0,1,\dots,s-1$. Since  $p\in \mcV(C)\smallsetminus \mcV_{N_0}$ it follows that $e_{p,\ell}\ge\val(p)$ for all $\ell\in\BN$.

Suppose $e_{v_1,N_0}<\val(v_1)$. 
As $p\in\mcF(E_{N_0})$, either: 1) $e_{v_1,N_0+1}\ge e_{v_1,N_0}+1$ or 2) there exists $p_{1}\in \alpha_1^{\circ}$ such that $e_{p_{1},N_0+1}\ge e_{p_{1},N_0}+1$. Let $s_1\coloneqq d_{\alpha_1}(p, p_{1})$. Now, we can see that either: 
1) $e_{v_1,N_0+2}\ge e_{v_1,N_0+1}+1$ or 2) there exists $p_{2}\in \alpha_1^{\circ}$ such that $e_{p_{2},N_0+2}\ge e_{p_{2},N_0+1}+1$ and $d_{\alpha_1}(p_2,v_1)=d_{\alpha_1}(p_1,v_1)-s_1$. Thus, as the length of $\alpha_1$ is finite, there exists $\ell_1\in\BN$ such that $e_{v_1,N_0+\ell_1}>e_{v_1,N_0}$; namely $\ell_1\le \big\lceil \text{length}(\alpha_1)/{s_1} \big\rceil$.

With an analogous argument, we show that there exists $\ell_2\in\BN$ such that $e_{v_2, N_0+\ell_2}>e_{v_2, N_0}$. Now, repeating the idea, we conclude that there exists $\ell_s$ such that 
\[
    e_{q,N_0+\ell_s}>e_{q, N_0}.
\]
Using \ref{prop: ES2} and this last inequality, we obtain Inequality \eqref{ine: inc_neg_degree}.
\end{proof}

In analogy with the combinatorial case, \autoref{definition: Volume of a divisor}, we have following Definition:

\begin{defin}\label{def: Vol_C(D)}
    The tropical \emph{volume} of a divisor $D$ on $C$ is defined as 
    \[
        \vol_{C}^{\BT}(D)\coloneqq\limsup_{\ell\to\infty} \frac{r(\ell D)}{\ell}.
    \]
\end{defin}

\begin{rem}\label{rem:tropi-implies-combi}
    \autoref{lemma: lower_bound_trop_curves} implies \autoref{lemma: lower_bound} in the following way:  In \cite[Def.\ 1.3]{Amini_Caporaso_2013} the authors, from a multigraph $G$, construct a loopless multigraph $\widehat{G}$ by adding one vertex to the interior of each loop $e\in E(G)$. They also point out that the inclusion map $\sigma:\Div(G)\to\Div(\widehat{G})$ is such that $r_{\rBN,G}(D)\ge r_{\rBN,\widehat{G}}\big(\sigma(D)\big)$. This last inequality, combined with \cite[Thm.\ 1.3]{Hladk_2013} allows us to conclude that \autoref{lemma: lower_bound_trop_curves} indeed implies \autoref{lemma: lower_bound}. As a consequence, it also follows that \autoref{prop: deg=asymptotic dim} is implied by the following Theorem: 
\end{rem}

\begin{thm}\label{prop: deg=asymptotic dim trop. curve}
    Let $D \in \Div(C)$, then $\vol_C^{\BT}(D)=\max\{\deg(D),0\}$.
\end{thm} 

\begin{proof}
    If $\deg(D)\leq0$, then by \autoref{rem: prop_trop_linear_series} we have $-1\leq r(\ell D) \leq 0$ for all $\ell\in\BN$. Thus, it follows that  $\vol_C^{\BT}(D)=0$.
    
    By \autoref{lemma: lower_bound_trop_curves} it follows that $r(\ell D)\ge \ell\deg(D) - c_C$, for all $\ell\in\BN$. Now, if $\deg(D)>0$, then by \autoref{rem: prop_trop_linear_series} we have $\deg(\ell D)\ge r(\ell D)$ for all $\ell\in\BN$. Thus, 
    \[
        \frac{\ell\deg(D)}{\ell} \ge \frac{r(\ell D)}{\ell}\ge \frac{\ell\deg(D)-c_C}{\ell}
    \]
    for all $\ell>0$.  We are done by taking  $\limsup$ with respect to  $ \ell\to\infty$.
\end{proof}

\section{Asymptotics of the independence rank of linear series on tropical curves}\label{Sec: JP rank}

We recall the \emph{independence rank}, a notion arising from tropical linear algebra, and compare it with the Baker–Norine rank. We study its algebraic and combinatorial properties, emphasizing key differences such as the failure of subadditivity. Despite these differences, we prove that both ranks have the same asymptotic behavior. This leads to a unified notion of tropical volume and an asymptotic Riemann--Roch theorem.

\medskip

The \emph{tropical semiring} is the triple $(\BT,\oplus,\odot)$, where $\BT\coloneqq \R\cup\{\infty\}$, 
    \[
        \begin{tikzcd}[row sep=0pt,/tikz/column 1/.append style={anchor=base east}
     ,/tikz/column 2/.append style={anchor=base west}]
            \oplus:\BT\times\BT \arrow[r] & \BT \\
            (a,b) \arrow[r, mapsto] & \min\{a,b\}
        \end{tikzcd}
        \quad \text{ and } \quad
        \begin{tikzcd}[row sep=0pt, /tikz/column 1/.append style={anchor=base east}
     ,/tikz/column 2/.append style={anchor=base west}]
            \odot:\BT\times\BT \arrow[r] & \BT \\
            (a,b) \arrow[r, mapsto] & a+b
        \end{tikzcd}
    \]
    are called \emph{tropical addition} and \emph{tropical multiplication}, respectively. We assume $a\oplus\infty=a$ and $a\odot\infty=\infty$ for all $a\in\BT$. We refer the reader to \cite[\S\;1.2]{giansiracusa2026} for more details concerning the tropical semiring $\BT$ and its different models. 

A \emph{$\BT$-module} (or \emph{tropical module}) is an abelian monoid $(M,\oplus)$ with identity $0_{M}$, and a map $\BT\times M \to M$ with $(t,m) \mapsto t m$
satisfying the defining axioms of a module over a ring, but with the tropical operations, i.e.\ for all $t, t'\in\BT$ and all $m,m'\in M$:
\begin{enumerate}[label = (TM\arabic*)]
    \item\label{axiom: trop_mod1} $(t\odot t') m = t(t' m)$;
    \item\label{axiom: trop_mod2} $t(m\oplus m')=tm\oplus tm'$, and  $(t\oplus t')m=tm\oplus t'm$;
    \item $0 m = m$ and $t 0_M = 0_M = \infty m$.
\end{enumerate}

See \cite[pg.\ 101]{Golan_2003} and \cite[pg.\ 1330]{Sumi_2021} for more details about $\BT$-modules. For instance, the set of tuples $\BT^n$ with componentwise tropical operations is a $\BT$-module. Let $S\subseteq M$ be a subset. We write $\langle S \rangle$ to represent the smallest $\BT$-submodule of $M$ that contains $S$.

Let $\PL(\Gamma)$ denote the group of continuous piecewise-linear functions with integer slopes on the compact tropical curve $\Gamma$.
Given functions $\varphi_1,\dots,\varphi_r$ in $\PL(\Gamma)$, the function $\min_{i}(\varphi_i)\in\PL(\Gamma)$ is defined as $\min_{i}(\varphi_i)(x)=\min_{i}\!\big(\varphi_i(x)\big)$.

\begin{defin}
 A tuple $(\varphi_1, \dots,\varphi_r)$ of elements in $\PL(\Gamma)$ is \emph{tropically dependent}\  if there is a tropical linear combination 
\[
    \bigoplus_{i=1}^r a_i\odot\varphi_i 
    =\min_{i}(a_i+\varphi_i) 
    \quad \text{ where } \quad (a_1,\dots,a_r)\in\BT^{r}\smallsetminus \{(\infty,\dots,\infty)\},
\]
such that the minimum is achieved at least twice at every point $x \in \Gamma$. If $(\varphi_0,\dots,\varphi_r)$ is not tropically dependent, we say that it is \emph{tropically independent}. Given a $\BT$-module $\Sigma$, its \emph{independence rank} $r_{\ind}(\Sigma)$ is the cardinality of the largest tropically independent subset of $\Sigma$. 
\end{defin}
 In other words, $(\varphi_0, \dots,\varphi_r)$ is tropically dependent if and only if 
\[
    \bigoplus_{i=0}^r a_i\odot\varphi_i 
    =\bigoplus_{i\neq j} a_i\odot\varphi_i \quad \text{ for each } \quad j=0,\dots,r. 
\]

For $\varphi\in\PL(\Gamma)$, its divisor $\ddiv(\varphi)$ is defined as in \cite[pg.\ 212]{Mikhalkin_Zharkov_2008}.
Let $D\in\Div(\Gamma)$. We define the following $\BT$-module 
\[
    R(D) \coloneqq \{\varphi \in \PL(\Gamma) \mid D + \ddiv(\varphi) \ge 0\}.
\]

\begin{defin}[Baker--Norine rank of a tropical submodule]
Let $\Sigma \subseteq R(D)$ be a $\BT$-submodule. Its \emph{Baker--Norine rank} $r_{\rBN}(\Sigma)$ is the largest $r\in\BN$ such that, for every $E\in\Div_{+}^{r}(\Gamma)$, there exists $\varphi \in \Sigma$ such that $D - E + \ddiv(\varphi) \geq 0$. If such maximal $r$ does not exist, then $r_{\rBN}(\Sigma)\coloneqq -1$.
\end{defin}
 \begin{rem}\label{rem: BN_rank_mod =BN+rank_div}
     Note that if $\Sigma = R(D)$, then we recover \autoref{def: complete linear series for metric graphs}. We will call $\BT$-submodules $\Sigma \subseteq R(D)$ also \emph{linear series}. 
 \end{rem}

The next Proposition proves a similar result to \cite[Prop.\ 8.3]{trop_linear_series_matroids_25}. The proof is similar as well, which we include for completeness. We emphasize that our result holds for any $\BT$-submodule of $R(D)$.

\begin{prop}\label{equ: Inequality rk_ind and rk_BN}
Let $\Sigma\subseteq R(D)$ be a $\BT$-submodule, then $r_{\ind}(\Sigma)\ge r_{\rBN}(\Sigma)+1$.
\end{prop}

The geometric notions used in the next proof are easily understood if one makes a parallel with a real smooth manifold of dimension $1$, or more specifically, identify a small interval $I\subset \Gamma$ with an interval in the real line. For formal definitions, we refer the reader to \cite[\S\ 2]{trop_linear_series_matroids_25}.

\begin{proof}    
    Let $\zeta$ be a tangent vector on $\Gamma$ at $v$, and 
    $
        s_{v,\zeta}(\Sigma)\coloneqq \{ {\partial\varphi}/{\partial\zeta}(v) \in\BZ \mid \varphi\in\Sigma\} 
    $
    the set of slopes at $v$ along $\zeta$.  
    Let $n_{v,\zeta}\coloneqq\#s_{v,\zeta}(\Sigma)$, and fix $\varphi_{v,\zeta,1},\dots,\varphi_{v,\zeta,n_{v,\zeta}} \in  \Sigma$ that realize the slopes in $s_{v,\zeta}(\Sigma)$.
    With abuse of notation, let $I_{v,\zeta}\coloneqq [v,w_\zeta)$, with a fixed orientation induced by $\zeta$, and $I_{v,\zeta}\cap\Supp(D)=\varnothing$ be such that $\varphi_{v,\zeta,i}|_{I_{v,\zeta}}$ has constant slope for $i=1,\dots,n_{v,\zeta}$. 
    Thus, for any $u\in I_{v,\zeta}$, it follows that $s_{u,\eta}(\Sigma)\supseteq s_{v,\zeta}(\Sigma)$, for a tangent vector $\eta$ at $u$ with the same direction as $\zeta$. 
    Choose $u\in I_{v,\zeta}$ with maximal $n_{u,\eta}$, and fix $I_{u,\eta}\coloneqq[u,w_\eta)\subseteq I_{v,\zeta}$.
    Then, for any $u'\in I_{u,\eta}$ it follows that $s_{u',\eta'}(\Sigma)=s_{u,\eta}(\Sigma)$. Let $r\coloneqq r_{\rBN}(\Sigma)$, and $E \coloneqq x_1 + \cdots + x_r \in \Div(\Gamma)$ where the $x_i$ are pairwise distinct points in $(u,w_{\eta})$. 
    Thus, there exists $\varphi\in\Sigma$ such that $D+\ddiv(\varphi)\ge E$. As $\varphi|_{I_{u,\eta}}$ has at least $r+1$ distinct slopes, it follows that $n_{u,\eta}\ge r+1$. 
    
    Since any set of functions in $\Sigma$ with pairwise different slopes at $u$, along $\eta$, is tropically linearly independent, it follows that $r_{\ind}(\Sigma)\ge n_{u,\eta} \ge r_{\rBN}(\Sigma) +1$.
\end{proof}

\begin{rem}\label{def: Trop_Lin_Ser_JP}
    Finitely generated $\mathbb{T}$-submodules $\Sigma \subseteq R(D)$ for which the inequality in \autoref{equ: Inequality rk_ind and rk_BN} is actually an equality have been studied in the literature and are known as \emph{tropical linear series} (TLS); see \cite[Def.\ 1.4]{trop_linear_series_matroids_25}. Their geometry has been related to the combinatorial geometry of matroids. However, in general, complete linear series are not tropical linear series; see \cite[\S\ 1.4]{Jensen_Payne_2022} and \autoref{R(D) is not a TLS}. One of the goals of this work is to prove that, nevertheless, they satisfy an symptotic Riemann--Roch theorem; see \autoref{thm:asymptotic-rr}. Moreover, in \autoref{exa: R-R does not holt for r_ind} we show that even if a complete linear series is tropical, it may fail to satisfy a Riemann--Roch theorem with respect to the independence rank. We emphasize that most of our results are for a general $\BT$-submodule $\Sigma\subseteq R(D)$, and do not rely on the condition $r_{\ind}(\Sigma)=r_{\rBN}(\Sigma)+1$ nor on the finite generation hypothesis. 
\end{rem}

Let $\Sigma\subseteq R(D)$ be a $\BT$-submodule and $\BG_{\BT}\coloneq\BT\smallsetminus\{\infty\}$. The  \emph{tropical projectivization} $|\Sigma|$ of $\Sigma$ is the quotient $(\Sigma - \{0_M\})/\hspace{-3pt}\sim$ where $\varphi\sim \varphi'$ if and only if there exists $t\in\BG_{\BT}$ such that $\varphi =t \odot\varphi'$.
If $\Sigma$ is  finitely generated, then by \cite[Prop.\ 3.4]{trop_linear_series_matroids_25} it follows that $|\Sigma|$ is a closed polyhedral subset of $\Sym^d(\Gamma)$, where $d=\deg(D)$. 
We show, using an example by D.~Jensen and S.~Payne, that if $\Sigma\subset R(D)$ is not finitely generated, then $|\Sigma|$ need not be closed.

\begin{example}[$|\Sigma|$ may fail to be closed]\label{exa: trop_edge}
        Let $\Gamma$ be the compact tropical curve as in \autoref{fig: one edge curve}.         
        \begin{figure}[ht!]
            \centering
            \begin{tikzpicture}
                \pgfmathsetmacro{\a}{0}
                \pgfmathsetmacro{\b}{1}
                \node[] at (-1,.08) {$\Gamma:$};
                \filldraw[black] (\a,0) circle (2pt) node[xshift=-9pt] {$v_1$};
                \filldraw[black] (\b,0) circle (2pt) node[xshift=9pt] {$v_2$};
                \draw[] (\a,0) -- (\b,0);
            \end{tikzpicture}
            \caption{Compact tropical curve that consists of one edge.}
            \label{fig: one edge curve}
        \end{figure}
        In general, the $\BT$-module $R(D)$ is finitely generated; see \cite[Cor.\ 9]{Haase_Musiker_2012}.  
        Now, let $D=v_1\in\Div(\Gamma)$, in this case $R(D)=\langle \varphi_0,\varphi_1\rangle$ where $\varphi_0=0$ and $\varphi_1(x) = x$. As $r_{\ind}(D)=2$ and $r_{\rBN}(D)=1$, it follows that $R(D)$ is a tropical linear series.
        
        Let $\Sigma\subseteq R(D)$ be the following $\BT$-submodule 
        \[
            \Sigma\coloneqq\{\varphi \in R(D)\mid \varphi(v_2)-\varphi(v_1)<1\} \cup \{\infty\}.
        \]
        Since there does not exist $\varphi\in\Sigma$ such that $D+\ddiv(\varphi)\ge v_2$, it follows that $r_{\rBN}(\Sigma)=0$. 
        As $\varphi_0$ and $\frac{1}{2}\odot\varphi_0\oplus \varphi_1 \in \Sigma$ are tropically independent, we have $r_{\ind}(\Sigma)> r_{\rBN}(\Sigma)+1$, in particular,  $\Sigma$ is not a tropical linear series.
        Moreover, $\Sigma$ is not a finitely generated submodule of $R(D)$, and we can identify $|R(D)|$ with $\Gamma$, and $|\Sigma|=\Gamma\smallsetminus\{v_2\}$ which is not closed in $|R(D)|$.
    \end{example}

\begin{example}[$R(D)$ may fail to be a tropical linear series, {\cite[Exa.\ 8.1]{trop_linear_series_matroids_25}}]\label{R(D) is not a TLS}
    Let $\Gamma$ be the compact tropical curve as in \autoref{fig: barbel graph}

    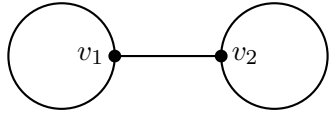
\begin{figure}[ht!]
        \centering
        \begin{tikzpicture}[thick]
            \pgfmathsetmacro{\a}{.7}
            \draw (0,0) circle (\a);
            \draw (4*\a,0) circle (\a);
            \draw (\a,0) -- (3*\a,0);
    
            \filldraw[black] (\a,0) circle (2pt) node[xshift=-9pt] {$v_1$};
            \filldraw[black] (3*\a,0) circle (2pt) node[xshift=9pt] {$v_2$};
        \end{tikzpicture}
        \caption{The compact tropical curve $\Gamma$.}
        \label{fig: barbel graph}
    \end{figure}

    Let $K_{\Gamma}= v_1  + v_2$ be its canonical divisor. Directly from the definition it follows that $r_{\rBN}(K_{\Gamma})=1$. On the other hand, let $\varphi_1, \varphi_2, \varphi_3 \in \PL(\Gamma)$ be such that 
    each $\varphi_i$ is constant equal to $\varphi_{i}(v_j)$ in each loop, and $\varphi_i$ has constant slope equal to $i-2$ in the edge, for each $i=1,2,3$. Thus, it is clear that $\varphi_i\in R(K_\Gamma)$ and also that $\varphi_1,\varphi_2$ and $\varphi_3$ are tropically linearly independent, that is, $r_{\ind}(K_\Gamma)\ge 3$. As $r_{\ind}(K_\Gamma) > r_{\rBN}(K_\Gamma)+1$ we conclude that $R(K_\Gamma)$ is not a tropical linear series.    
\end{example}

    Let $D$ be a divisor on a tropical curve. Directly from the definition of the independence rank, we see that if $\Sigma\subseteq\Sigma'$ are submodules of $R(D)$, then $r_{\ind}(\Sigma) \le r_{\ind}(\Sigma')$. In particular, since $r_{\ind}(D)$ is finite, see \autoref{thm: upper bound}, it follows that $r_{\ind}(\Sigma)$ is finite for any submodule $\Sigma$ of $R(D)$.

    For the sake of completeness, we now present an example by D.~Jensen and S.~Payne,
    where $\Sigma$ is not finitely generated. 

\begin{example}[$\Sigma$ may fail to be finitely generated]
     Let $\Gamma$ be the tropical curve that consists of two edges, of length $1$ each, and three vertices, that we identify with the interval $[-1,1]\subset \BR$; see \autoref{fig: graph 3_v_2e}.
    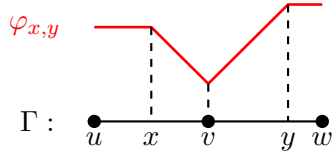
\begin{figure}[ht!]
        \centering
        \begin{tikzpicture}[line width=1pt]
        \pgfmathsetmacro{\a}{1.5}
        \pgfmathsetmacro{\b}{.5}
    
    \coordinate (v1) at (\a,0);
    \coordinate (v2) at (2*\a,0);
    \coordinate (v3) at (3*\a,0);
    
    \draw[thick, shorten <=0pt, shorten >=0pt] (v1) -- (v2) -- (v3);

    \draw[red] (\a, .5*\a + \b) -- (1.5*\a,.5*\a + \b) -- (2*\a,\b) -- (2.7*\a,.7*\a + \b) -- (3*\a,.7*\a + \b);

    \draw[dashed, thick] (1.5*\a,.5*\a + \b) -- (1.5*\a,0) node[below] {$x$};
    \draw[dashed, thick] (2*\a,\b) -- (v2);
    \draw[dashed, thick] (2.7*\a,.7*\a + \b) -- (2.7*\a,0) node[below] {$y$};

    \node[] (g) at (.5*\a,0) {$\Gamma:$};
    \node[left, red] (phi) at (.8*\a,.5*\a + \b) {$\varphi_{x,y}$};
    
    \filldraw[] (v1) circle (2pt) node[below] {$u$};
    \filldraw[] (v2) circle (2pt) node[below] {$v$};
    \filldraw[] (v3) circle (2pt) node[below] {$w$};
    
\end{tikzpicture}
        \caption{The curve $\Gamma$ and the functions $\varphi_{x,y}$.}
        \label{fig: graph 3_v_2e}
    \end{figure}
    
Let $\varphi_{x,y}$ be a function with slope $1$ from $v\coloneqq0$ to $x$ and from $v$ to $y$, as in \autoref{fig: graph 3_v_2e}.
Let $D= 2v$, and 
\[
    \Sigma\coloneqq \{ \varphi_{x,y} \in R(D) \mid y-x\le1\}
\]
be a submodule of $R(D)$.

Note that $r_{\rBN}(D)=2$ and $r_{\ind}(D)=3$, that is, $R(D)$ is a tropical linear series. More specifically, as $\Sigma\subseteq R(D)$ it follows that $r_{\ind}(\Sigma)\le 3$, and as $\varphi_{0,1},\, \varphi_{\frac{1}{2},\frac{1}{2}}$, and $\varphi_{1,0} \in \Sigma$ are tropically independent, we have $r_{\ind}(\Sigma)=3$. In contrast, $r_{\rBN}(\Sigma)=1$, that is, there exists no $\varphi_{x,y}\in\Sigma$ such that $D+\ddiv(\varphi_{x,y})\ge 2 u$. This shows that $\Sigma$ is not a tropical linear series. 

Now we argue that $\Sigma$ is not finitely generated. Let $I=[1,n]\subset\BN$ and suppose that there are  $\varphi_i\coloneqq \varphi_{x_i,y_i}$, for some $x_i, y_i$ and $i\in I$ such that $\Sigma=\langle \varphi_i \mid i\in I\rangle$. Then, for any $\varphi_{x,y}\in\Sigma$ there exist $a_i\in \BR$ such that $\varphi_{x,y}=\bigoplus_{i\in I} a_i\odot\varphi_i$. In this case, $x\ge x_i$ and $y\le y_i$ for all $i\in I$. In particular, this shows that not all functions $\varphi_{x,y}\in\Sigma$ with $x+y=1$ can be expressed as a tropical linear combination of the $\varphi_i$. This shows that $\Sigma$ is not finitely generated.
\end{example}

We note that the independence rank in general does not satisfy an inequality along the lines of Inequality \eqref{ineq: dim|D|_is_convex}. Indeed, consider the compact tropical curve $\Gamma$ that consists of only one edge, an analogous, but easier case as in \autoref{fig: graph_Gamma_tree}. In particular, since $\Gamma$ is a tree, the BN-rank depends only on the degree, thus $\Pic(\Gamma)=\Z$. For a fixed $d\ge 1$, let $D\in\Div^{d}(\Gamma)$. Thus, for $\ell \in \BN$
    \begin{align*}
        r_{\ind}(\ell D) & = r_{\rBN}(\ell D) + 1 = r_{\rBN}(\ell'D) + r_{\rBN}(\ell'' D) +1 < r_{\ind}(\ell'D) + r_{\ind}(\ell'' D),
    \end{align*}
    for any non-negative integers $\ell',\ell''$ with $\ell=\ell'+\ell''$. This shows that $r_{\ind}$ does not satisfy an analog of the Inequality \eqref{ineq: dim|D|_is_convex}.

\begin{prop}\label{thm: upper bound}
        Let $D\in\Div(\Gamma)$ with $\deg(D)\ge - 1$ and $f_1,\dots,f_m\in R(D)$ tropically independent. Then, $\deg(D)+1 \ge m$. In particular, $\deg(D) + 1 \ge r_{\ind}(\Sigma)$.
\end{prop}

\begin{proof}
    Let $\Sigma=\langle f_1,\dots,f_m\rangle \subseteq R(D)$ be a $\BT$-submodule, and $|\Sigma|$ its projectivisation. As $|\Sigma|\subseteq \Sym^d(\Gamma)$ with $d=\deg(D)$, it follows that $\dim|\Sigma|\le d$. On the other hand, by \cite[Prop.\ 3.6]{trop_linear_series_matroids_25} we have $r_{\ind}(\Sigma)=\dim|\Sigma|+1$. Hence 
    \[
        r_{\ind}(\Sigma) = \dim|\Sigma
        |+1\le \deg(D)+1.
    \]
    As the $f_i$ are tropically independent, it follows that 
    $m\le r_{\ind}(\Sigma)\le \deg(D)+1$.
\end{proof}

Given linear series (i.e.\ $\BT$-submodules) $\Sigma_1 \subseteq R(D_1)$ and $\Sigma_2 \subseteq R(D_2)$, define the multiplication map 
\[
    \begin{tikzcd}[row sep=0pt,/tikz/column 1/.append style={anchor=base east}
     ,/tikz/column 2/.append style={anchor=base west}]
        \mu \colon \Sigma_1 \times \Sigma_2 \arrow[r]  & R(D_1 + D_2), \quad 
         (f_1,f_2) \arrow[r,mapsto]  & \mu (f_1 , f_2)\coloneqq f_1 \odot f_2.
    \end{tikzcd}
\]
We use $\Sigma_1\odot\Sigma_2$ to denote the $\BT$-submodule generated by $\mu(\Sigma_1\times\Sigma_2)$, that is 
\[
    \Sigma_1\odot\Sigma_2\coloneqq \langle \mu(\Sigma_1\times\Sigma_2)\rangle\subseteq R(D_1+D_2).
\]
For any natural number $\ell$, we write $\ell\Sigma$ to represent $\ell$ times the product of $\Sigma$. 

\begin{defin}\label{def: trop+ind_volumes}
 Let $\Sigma\subseteq R(D)$ be a 
 $\BT$-submodule.
 \begin{enumerate} 
 \item[(i)] We define the \emph{independence  volume} and the \emph{BN-volume} of $\Sigma$ as 
\[
    \vol_{\ind, \Gamma}^{\BT}(\Sigma)\coloneqq \limsup_{\ell\to\infty}\frac{r_{\ind}(\ell \Sigma)}{\ell} \quad \text{ and } \quad \vol_{\rBN, \Gamma}^{\BT}(\Sigma)\coloneqq \limsup_{\ell\to\infty}\frac{r_{\rBN}(\ell \Sigma)}{\ell},
\]
respectively.   

\item[(ii)] 
Similarly, we define the \emph{independence  volume} and the \emph{BN-volume} of $D$ as 
\[
    \vol_{\ind, \Gamma}^{\BT}(D)\coloneqq \limsup_{\ell\to\infty}\frac{r_{\ind}(\ell D)}{\ell} \quad \text{ and } \quad \vol_{\rBN, \Gamma}^{\BT}(D)\coloneqq \limsup_{\ell\to\infty}\frac{r_{\rBN}(\ell D)}{\ell},
\]
respectively.   
\end{enumerate}
\end{defin}
    \begin{rem}\label{rem: vol_D_and_t_mod}
        \begin{enumerate}
            \item[(i)]
        Note that the volume $\vol_{\rBN,\Gamma}^{\BT}(D)$ agrees with the volume in \autoref{def: Vol_C(D)}. Furthermore, it will follow from \autoref{thm: deg=asymptotic ind_dim trop} that the independence volume and the BN-volume of $D$ agree.
        When $\Sigma = R(D)$, in general $\vol_{\alpha,\Gamma}^{\BT}(D)\ge\vol_{\alpha,\Gamma}^{\BT}\!\big(R(D)\big)$ for $\alpha\in\{\rBN,\ind\}$ and a strict inequality can occur, as it possible to see with the divisor $D=v$ on the ``lollipop" curve; see \autoref{exa: r_ind neq r_BN + 1}.   
    
       \item[(ii)]  The algebro-geomeric analogue of $\vol_{\rBN,\Gamma}^{\BT}\!\big(D\big)$ is the volume of graded linear series $\bigoplus_{\ell}H^0\big(X,\mathcal{O}(\ell D)\big)$, whereas the analogue of $\vol_{\rBN,\Gamma}^{\BT}\!\big(R(D)\big)$ is the volume of the subalgebra generated in degree one. These are also normally not the same. In toric geometry, this is related to the notion of \emph{normal polytopes}; we refer to \cite[\S\ 2.2.]{CLS} for more details.
       \end{enumerate}
\end{rem}
 
We have the following Theorem: 
\begin{thm}\label{thm: deg=asymptotic ind_dim trop}
 Let $D\in\Div(\Gamma)$ and let $\Sigma \subseteq R(D)$ a $\BT$-submodule. 
 \begin{enumerate}
     \item[(i)] For $D$, we have 
    \[
        \vol_{\ind, \Gamma}^{\BT}(D) = \vol_{\rBN,\Gamma}^{\BT}(D) = \max\{\deg(D),0\}.
    \]
    \item[(ii)] 
    For $\Sigma$, we have  the following sequence of inequalities
    \begin{equation}\label{equ: sharp+inequality}
        \max\{\deg(D),0\}\ge \vol_{\ind, \Gamma}^{\BT}(\Sigma)\ge \vol_{\rBN, \Gamma}^{\BT}(\Sigma) \ge \max\{r_{\rBN}(\Sigma),0\}.
    \end{equation}

Moreover, all inequalities are sharp. 
    \end{enumerate}
\end{thm}
\begin{proof}
      
    We first prove (ii). If $\deg(D)<0$, then $\vol_{\ind, \Gamma}^{\BT}(\Sigma)=\vol_{\rBN, \Gamma}^{\BT}(\Sigma)=0$.
    
    If $\deg(D) \ge 0$, we use \autoref{thm: upper bound} and \autoref{equ: Inequality rk_ind and rk_BN}, to write $\deg(\ell D)+1 \ge r_{\ind}(\ell\Sigma)\ge r_{\rBN}\ell\Sigma)+1$. Multiplying by $1/\ell$ and taking the limit superior, we obtain 
    \[
        \deg(D)\ge \vol_{\ind, \Gamma}^{\BT}(\Sigma)\ge \vol_{\rBN, \Gamma}^{\BT}(\Sigma).
    \]

Moreover, in (the proof of) \autoref{equ: Inequality rk_ind and rk_BN} we show that for a small enough intervel $I$, the set of slopes $s_{v,\zeta}(\Sigma)=\{ s_1 < \cdots < s_{m}\}$ has cardinality at least $r_{\rBN}(\Sigma)+1$ i.e.\ $m\ge r_{\rBN}(\Sigma)+1$. Given two finite sets $A,B\subset \BZ$, it follows that $\#(A+B)\ge \#A + \#B -1$, where $A+B$ stands for the Minkowski sum. By induction and this last inequality, it follows that $\#s_{v,\zeta}(\ell \Sigma)\ge \ell r_{\rBN}(\Sigma)+1$, in particular $r_{\ind}(\ell \Sigma)\ge \ell r_{\rBN}(\Sigma)+1$. Thus, $\vol_{\ind,\Gamma}^{\BT}(\Sigma)\ge r_{\rBN}(\Sigma)$. 

 If $r_{\rBN}(\Sigma)<0$, then $0=\vol_{\rBN,\Gamma}^{\BT}(\Sigma)>r_{\rBN}(\Sigma)$. Now, assume $r\coloneqq r_{\rBN}(\Sigma)\ge0$  and let $E\in\Div_{+}^{\ell r}(\Gamma)$ be arbitrary. Write $E=E_1 + \cdots + E_{\ell}$ with $E_i\in \Div_{+}^{r}(\Gamma)$ for each $i$. As $r=r_{\rBN}(\Sigma)$, there is $\varphi_{i}\in \Sigma$ with $D +\ddiv(\varphi_i) \ge E_i$ for each $i=1,\dots,\ell$. Thus, $\varphi_1\odot\cdots\odot\varphi_\ell\in \ell\Sigma$ and $\ell D +\ddiv(\varphi_1\odot\cdots\odot\varphi_\ell)\ge \sum E_i$. In other words $r_{\rBN}(\ell \Sigma) \ge \ell r_{\rBN}(\Sigma)$, which implies that $\vol_{\rBN,\Gamma}^{\BT}(\Sigma)\ge r_{\rBN}(\Sigma)$.

Thus, we can write 
\begin{equation*}\label{equ: ine_vol_t_mod}
    \max\{\deg(D),0\}\ge \vol_{\ind, \Gamma}^{\BT}(\Sigma)\ge \vol_{\rBN, \Gamma}^{\BT}(\Sigma) \ge \max\{r_{\rBN}(\Sigma),0\}.
\end{equation*}

 The Inequalities in \eqref{equ: sharp+inequality} are sharp, indeed if $\Gamma$ and $D=v_1$ are as in \autoref{exa: trop_edge}, then all Inequalities in \eqref{equ: sharp+inequality} are equalities, namely $\deg(D)=r_{\rBN}\!\big(R(D)\big)=1$. If $\Gamma$ and $D=v$ are as in \autoref{Fig: Lollipop}, then  
\[
    1=\max\{\deg(D),0\}=\vol_{\ind, \Gamma}^{\BT}(\Sigma)> \vol_{\rBN, \Gamma}^{\BT}(\Sigma) = \max\{r_{\rBN}(\Sigma)=0,0\}.
\]

In contrast, for a divisor $D\in \Div(\Gamma)$, we use the same idea and steps as before to obtain
\[
   \max\{\deg(D),0\}\ge \vol_{\ind, \Gamma}^{\BT}(D)\ge \vol_{\rBN, \Gamma}^{\BT}(D) = \max\{\deg(D),0\},
\]
where the equality above on the right-hand side follows from \autoref{rem: BN_rank_mod =BN+rank_div} and \autoref{prop: deg=asymptotic dim trop. curve}.
\end{proof}

We use \autoref{thm: deg=asymptotic ind_dim trop} to write $\vol_{\ind, \Gamma}^{\BT}(D)=\vol_{\rBN,\Gamma}^{\BT}(D)$ simply as $\vol_{\Gamma}^{\BT}(D)$ and call it the \emph{tropical volume} of $D$. \autoref{thm: deg=asymptotic ind_dim trop} also shows that the volume of a divisor satisfies all expected properties of the (classical) volume. For instance, it is linear, or more precisely $\vol_{\Gamma}^{\BT}(aD)=a\vol_{\Gamma}^{\BT}(D)$ for any $a\in\BN$, and $D\in\Div(\Gamma)$ is big if and only if $\deg(D)>0$. 

\medskip

Let $K_\Gamma\coloneqq \sum_{v\in\mcV(\Gamma)}\big(\hspace{-2pt}\val(v)-2\big)v$ be the \emph{canonical divisor} of $\Gamma$. The \emph{independence} and \emph{Baker--Norine} \emph{Euler characteristic} of $D \in\Div(\Gamma)$ is defined as 
\[
    \chi_{\ind}(D)\coloneqq r_{\ind}(D)-r_{\ind}(K_\Gamma-D) \quad \text{ and } \quad \chi_{\rBN}(D)\coloneqq r_{\rBN}(D)-r_{\rBN}(K_\Gamma-D),
\]
respectively.

\begin{example}\label{exa: r_ind neq r_BN + 1}
    Let $\Gamma$ be a tropical curve that is a tree; thus, $\Pic(\Gamma)=\Z$. Moreover, for any $D\in\Div(\Gamma)$ it is not difficult to verify that $r_{\ind}(D)=r_{\rBN}(D)+1$, that is, $R(D)$ is a tropical linear series. 

    Consider the ``simplest" case which is not a tree. Namely, the connected multigraph with one edge and one loop; see \autoref{Fig: Lollipop}.
    \begin{figure}[ht!]
        \centering
        \begin{tikzpicture}[thick]
            \pgfmathsetmacro{\a}{.7}
            \draw (4*\a,0) circle (\a);
            \draw (\a,0) -- (3*\a,0);
            \filldraw[black] (\a,0) circle (2pt);
            \filldraw[black] (3*\a,0) circle (2pt) node[xshift=8pt] {$v$};
        \end{tikzpicture}
        \caption{The ``lollipop" tropical curve $\Gamma$.}
        \label{Fig: Lollipop}
    \end{figure}
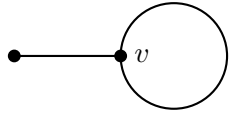
    
    Let $D = v \in \Div(\Gamma)$ and for $\ell\in\BN$ and consider $\ell D$. Analyzing the possible slopes on the edge, for a function $\varphi\in R(\ell D)$, which ranges from $0$ to $\ell$, we can see that $r_{\ind}(\ell D)=\ell + 1$. On the other hand,
    it follows that $r_{\rBN}(\ell D)=\ell -1$, for all $\ell\ge 1$. Thus, it follows that $r_{\ind}(\ell D)> r_{\rBN}(\ell D)+1$, for each $\ell\ge1$. This shows that $R(\ell D)$ is not a tropical linear series, for any $\ell \ge 1$. 

    Nevertheless, as stated in \autoref{thm: deg=asymptotic ind_dim trop} and \autoref{prop: deg=asymptotic dim trop. curve}, and via a direct calculation, it follows that 
    \[
        \vol_{\Gamma}^{\BT}(D)=  \limsup_{\ell\to\infty}\frac{r_{\ind}(\ell D)}{\ell} = \limsup_{\ell\to\infty}\frac{r_{\rBN}(\ell D)}{\ell} = \deg(D). 
    \]
    The complete linear series $\Sigma=R(\ell D)$
    is finitely generated. 
    Thus, for any $\ell \ge 2$ we have a $\BT$-module $\Sigma$ that is not a tropical linear series, that is, $r_{\ind}(\Sigma)> r_{\rBN}(\Sigma)+1$ and is finitely generated. 
\end{example}

This example illustrates that the volume of a divisor gives us the expected result, even if the linear series is not tropical.  
  
\begin{thm}[Asymptotic Riemann--Roch for tropical curves]\label{thm:asymptotic-rr}
    Let $D\in\Div(\Gamma)$ and $\iota$ an index in $\{\ind,\rBN\}$, then 
    \[
        \chi_{\iota}(\ell D)=\deg(D)\ell + o(1).
    \]
\end{thm}
\begin{proof}
    First consider $\deg(D)\ge0$. In this case $r_{\iota}(K_\Gamma-\ell D)$ is bounded. Thus, by \autoref{thm: deg=asymptotic ind_dim trop}
    \[
        \limsup_{\ell\to\infty}\frac{\chi_{\iota}(\ell D) - \deg(D)\ell}{\ell}=0.
    \]
    If $\deg(D)<0$, then $r_{\iota}(\ell D)$ is bounded. For a large enough fixed $\ell'\in \BN$ and all $\ell \in \BN$, it follows that $\deg\!\big((K_\Gamma-\ell'D) -\ell D\big)\geq -1$, and $r_{\rBN}(K_\Gamma -\ell'D)\ge0$ by \autoref{lemma: lower_bound_trop_curves}. Thus by \autoref{thm: upper bound} and \autoref{equ: Inequality rk_ind and rk_BN} it follows that
    \begin{equation*}
        \deg\!\big((K_\Gamma - \ell'D) - \ell D\big) + 1 \ge r_{\iota}\!\big((K_\Gamma - \ell'D) -\ell D\big) \ge r_{\rBN}(K_\Gamma-\ell'D) + r_{\rBN}(-\ell D).
    \end{equation*}
    The Inequalities above and \autoref{prop: deg=asymptotic dim trop. curve} imply that 
    \[
        \limsup_{\ell\to\infty}\frac{\chi_{\iota}(\ell D) - \deg(D)\ell}{\ell}=0.
    \]
    This finishes the proof.
\end{proof}

If a divisor $D$ in a tropical curve $\Gamma$ is such that, both $R(D)$ and $R(K_\Gamma-D)$ are tropical linear series, then by Riemann--Roch theorem for $\rBN$-rank, it follows that  
\[
    \chi_{\ind}(D)= \chi_{\rBN}(D) = \deg(D) - g_{\Gamma} + 1,
\]
that is, the independence rank also satisfies the Riemann--Roch theorem. 
Thus, the natural question is: If $R(D)$ is a tropical linear series, does it imply that $R(K_\Gamma - D)$ is a tropical linear series? The answer is negative, as we can see in the following Example:

\begin{example}[$r_{\ind}$ does not satisfy Riemman--Roch formula]\label{exa: R-R does not holt for r_ind}
    Let $\Gamma$ be the ``lollipop" tropical curve of \autoref{Fig: Lollipop}. Let $D=0\in\Div(\Gamma)$, and note that $r_{\ind}(D)=r_{\rBN}(D)+1=1$, that is, $R(D)$ is a tropical linear series. On the other hand, we know from \autoref{exa: r_ind neq r_BN + 1} that $r_{\ind}(K_\Gamma-D)=2>r_{\rBN}(K_{\Gamma})+1=1$, i.e.\ $R(K_\Gamma)$ is not a tropical linear series. Furthermore, 
    \[
        -1 = \chi_{\ind}(D) \neq \chi_{\rBN}(D)=\deg(D) - g_{\Gamma} + 1 = 0. 
    \]
\end{example}

This example shows that, although the Asymptotic Riemann--Roch theorem for the independence rank holds for any complete linear series---it does not need to be tropical---whereas the Riemann--Roch theorem for the independence rank does not hold in full generality, i.e.\ for any complete linear series $R(D)$.

\section{Tropicalization of linear series}\label{Sec: tropicalization_Lin_Ser}

This Section connects the tropical theory developed here with classical algebraic geometry via tropicalization. 
Using this framework, we show that tropical volume is compatible with tropicalization. This result supports the idea that tropical geometry captures essential asymptotic features of linear series on algebraic curves.

\subsection{Specialization map}\label{sec: Specialization_map}
Let $R$ be a DVR with fraction field $K$, and residue field $k$. Let $C$ be a smooth projective curve over $K$, and $\mathfrak{C}/R$ a semistable model of $C$, i.\,e.\,the speical fiber $\mathfrak{C}_k$ is reduced and has only ordinary double points as singularities. 

We recall that given the irreducible decomposition $X=\bigcup_{i=1}^{n} X_{i}$ of an algebraic curve $X$, its dual multigraph $\Gamma_{X}$ consists of a vertex for each irreducible component, and an edge for each intersection point; see \cite[pg.\ 249]{Harris_Morrison_1998} for more details concerning the dual multigraph of an algebraic curve, and \autoref{fig: degeneration_1} and \autoref{fig: degeneration_triangle} for illustrations. When the curve $X$ is clear from context, we write $\Gamma$ instead of $\Gamma_X$.

Moreover, let 
\[
    \begin{tikzcd}
        \rho:\Div(C) \ar[r, two heads] & \Div(\Gamma)
    \end{tikzcd}
\]
be the \emph{specialization map}, where $\Gamma$ is the dual multigraph of the special fiber $\mathfrak{C}_{k}$. 

The specialization map preserves the degree of divisors i.e.\ for any $D\in \Div(C)$, it follows that $\deg(D)=\deg\!\big(\rho(D)\big)$; see \cite[pg.\ 621]{Baker_2008}, \cite[pg.\ 7437]{Baker_Rabinoff_2015} or \cite{Baker_Jensen_2016} for more details regarding the specialization map and its properties. In particular, from \autoref{thm: deg=asymptotic ind_dim trop} we conclude that $\vol_{C}(D)=\vol_{\Gamma}^{\BT}\!\big(\rho(D)\big)$. 
As a consequence, the independence rank of $R\big(\rho(D)\big)$ recovers the well--known classical formula $\vol_{C}(D) =\max\{\deg(D),0\}$. We summarize these observations in the following proposition.
\begin{prop}\label{prop:specialization}
    Let $C$, $\Gamma$ and $\rho$ as above. Then for any divisor $D \in \Div(C)$ it holds that 
    \[
        \vol_C(D) = \vol_{\Gamma}^{\BT}\!\big(\rho(D)\big).
    \]
\end{prop}
\begin{proof} We have 
\begin{equation}\label{equation: deg(D_X)=vol(D_G)}
    \vol_{C}(D) = \max\{\deg(D),0\} = \max\!\big\{\!\deg\!\big(\rho(D)\big),0\big\} = \vol_{\Gamma}^{\BT}\!\big(\rho(D)\big).
\end{equation}
\end{proof}
We refer to \cite[\S\;2.2.C]{Lazarsfeld_2004_I} for more details regarding the volume of divisors. 

\begin{example}[Specialization of a divisor]\label{exa: vol(D)=vol(p(D))}
    Let $R=\BC[\![t]\!]$ be the ring of formal power series with coefficients in $\BC$. It is a DVR with uniformizing parameter $t$, maximal ideal $\langle t \rangle$, its fraction field is the field of \emph{Laurent series} $K=\BC(\!(t)\!)$ and its residue field is $k = \BC$.

    \begin{figure}[ht]
    \centering
    \begin{tikzpicture}[>=Latex, line cap=round, line join=round,
    x=1cm, y=1cm, scale=.65, line width=1pt, every node/.style={font=\normalsize}
        ]
        \def\W{9.0}
        \def\H{3.5}
        \def\ybase{-1.8}
        \def\xzero{2.0}
        \def\xn{6.6}

        \def\kappa{0.75}

        \def\fiberZeroCurv{0.25}
        \def\fiberZeroV{0.38}
        \def\fiberZeroBottom{0.10}
        \def\fiberZeroHeight{2.90}

        \def\fiberTCurvA{0.60} 
        \def\fiberTCurvB{0.45} 
        \def\fiberTBottom{0.85}
        \def\fiberTHeight{1.80}

        \coordinate (TL) at (0,\H);
        \coordinate (TR) at (\W,\H);
        \coordinate (BR) at (\W,0);
        \coordinate (BL) at (0,0);

        \coordinate (cR1) at ({\kappa*0.55},{\kappa*(-0.9)});
        \coordinate (cR2) at ({\kappa*0.55},{\kappa*( 0.9)});
        \coordinate (cB1) at ({\kappa*(-2.2)},{\kappa*(-0.35)});
        \coordinate (cB2) at ({\kappa*( 2.2)},{\kappa*(-0.35)});

        \draw (TR) .. controls ($(TR)+(cB1)$) and ($(TL)+(cB2)$) .. (TL) .. controls ($(TL)+(cR1)$) and ($(BL)+(cR2)$) .. (BL) .. controls ($(BL)+(cB2)$) and ($(BR)+(cB1)$) .. (BR) .. controls ($(BR)+(cR2)$) and ($(TR)+(cR1)$) .. (TR) -- cycle;

        \node at (4.5,\H + .4) {$V(xy - tz^2)\subset \BP^2_R$};

        \coordinate (Ztop) at (\xzero, {\fiberZeroBottom + \fiberZeroHeight});
        \coordinate (Zmid) at (\xzero, {\fiberZeroBottom + 0.5*\fiberZeroHeight});
        \coordinate (Zbot) at (\xzero, {\fiberZeroBottom});

        \draw (Ztop) .. controls ($(Ztop)+(-\fiberZeroCurv,-\fiberZeroV)$) and ($(Zmid)+(-\fiberZeroCurv, \fiberZeroV)$) .. (Zmid) .. controls ($(Zmid)+(\fiberZeroCurv,-\fiberZeroV)$) and ($(Zbot)+(\fiberZeroCurv, \fiberZeroV)$) .. (Zbot);

        \node at (\xzero - .6, {\fiberZeroBottom + 1.4}) {$\mathfrak{C}_{K}$};

        \coordinate (Ttop) at (\xn, {\fiberTBottom + \fiberTHeight});
        \coordinate (Tbot) at (\xn, {\fiberTBottom});
        \coordinate (Tmid) at (\xn, {\fiberTBottom + 0.5*\fiberTHeight});

        \draw ($(Ttop)+(-\fiberTCurvA,.5)$) .. controls ($(Tmid)+(-0.2,0.3)$) .. (Tmid) .. controls ($(Tmid)+(0.2,-0.3)$) .. ($(Tbot)+(\fiberTCurvA,.4)$);

        \draw ($(Ttop)+(\fiberTCurvB,-.4)$) .. controls ($(Tmid)+(0.2,0.3)$) .. (Tmid) .. controls ($(Tmid)+(-0.2,-0.3)$) .. ($(Tbot)+(-\fiberTCurvB,-.8)$);

        \node at ($(Tmid)+(0.8,0)$) {$\mathfrak{C}_{k}$};

        \draw (-0.1,\ybase) -- (\W+.3,\ybase);

        \fill (6.62,1.75) circle (3pt);
        \fill (\xzero,\ybase) circle (3pt);
        \fill (\xn,\ybase)    circle (3pt);

        \node[below=4pt] at (\xzero,\ybase) {$\langle 0\rangle$};
        \node[below=4pt] at (\xn,\ybase) {$\langle t\rangle$};

        \draw[dashed] (\xzero,{\fiberZeroBottom - .2}) -- (\xzero,\ybase);
        \draw[dashed] (\xn,{\fiberZeroBottom - .2}) -- (\xn,\ybase);

        \def\xSpec{-1.5}

        \coordinate (SpecPoint) at (\xSpec,\ybase);
        \node at (SpecPoint) {$\mathrm{Spec}\,R$};

        \node at (\xSpec,2.8) {$\mathbb{P}^{2}_{R}$};

        \draw[->] (\xSpec,2.2) -- (\xSpec,\ybase+0.5);
        \node[left] at (\xSpec,.7) {$\pi$};

        \end{tikzpicture}
        \begin{tikzpicture}[scale=1, line width=1pt, every node/.style={font=\normalsize}]
            \def\r{.8}
            \node[left]  (vx) at (-\r,0) {$v_{x}$};
            \node[right] (vy) at (\r,0) {$v_{y}$};
            \fill (-\r,0) circle (2pt);
            \fill (\r,0) circle (2pt);
            \node[right] at (260:1.5) {$\Gamma$};
            \node[right] at (270:3) {};
            \node[right] at (180:4) {};

            \draw (vx) -- (vy);
        \end{tikzpicture}
    \caption{The family of conics $\mathfrak{C}/R$ with smooth generic fiber $\mathfrak{C}_K$, and special fiber $\mathfrak{C}_k$, a nodal curve of compact type with two components (on the left), and the dual multigraph $\Gamma$ of the special fiber $\mathfrak{C}_{k}$ (on the right).}
    \label{fig: degeneration_1}
\end{figure}
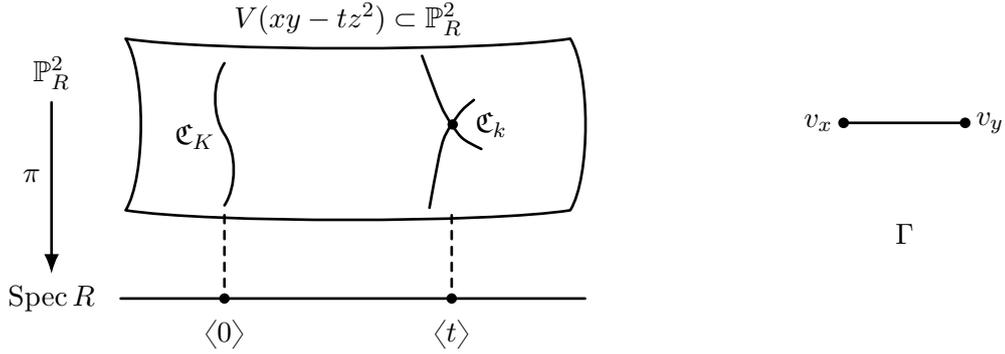

    The dual graph $\Gamma$ of $\mathfrak{C}_{k}$ has two vertices $v_x, v_y$, that correspond to the components $V(x),V(y)\subset \mathfrak{C}_k$ respectively, and one edge $e$, which corresponds to the nodal singularity. See the graph $\Gamma$ in \autoref{fig: degeneration_1} (on the right).

    Let $D_{K}\in\Div(\mathfrak{C}_{K})$ be the divisor cut out by the line $y-tx=0$, that is, $D_{K}=P_{K}+Q_{K}$, where $P_{K}=[1:t:1]$ and $Q_{K}=[1:t:-1]$. In the special fiber $\mathfrak{C}_k$, the divisor $D_{K}$ specializes to $D_k=P_k+Q_k \in\Div(\mathfrak{C}_k)$ where $P_{k}=[1:0:1]$ and $Q_{k}=[1:0:-1]$. As $\Supp(D_k)$ is contained in the component $V(y)\subset \mathfrak{C}_k$ it follows that 
    $D\coloneqq\rho(D_K)=2v_y\in\Div(\Gamma)$.

    We know that $r_{\ind}(D)=3$. More specifically, let $s\in[0,1]$ be a parameter of the edge $e$ that starts in $v_y$ and identify $e$ with $[0,1]$. Thus, it follows that $\varphi_0(s)=0,\varphi_1(s)=s$, and $\varphi_{2}(s)=2s$ in $R(D)\subset \PL(\Gamma)$ are tropically linearly independent. Furthermore, as $\Gamma$ is a tree, for any $\ell\in\BN$, with the same idea we conclude that $r_\ind(\ell D)= 2\ell + 1$, where $\varphi_{i}(s) = i s$, for $i=0,\dots,2\ell$ are tropically linearly independent. Thus, 
    \[
        \deg(D_{K}) = \vol_{\Gamma}^{\BT}(D)=2,
    \]
    which is a special case of Equations \eqref{equation: deg(D_X)=vol(D_G)}.
    That is, the degree of the divisor $D_K$ on the curve $C$ agrees with the tropical volume. This example is a slight modification and more detailed version of {\cite[Exa.\ 3.1]{Baker_Jensen_2016}}, where the authors consider the independence rank of the complete linear series $R(D)$.
\end{example}
\subsection{Extended tropicalization}
Let $K$ be an algebraically closed field, which is complete with respect to a nontrivial non-archimedean valuation $\nu \colon K^* \to \R$. The usual tropicalization map associates to a closed subvariety
$X$ of the torus $\BG_{K}^m = (K^*)^m$
the underlying set $\trop(X)$ of a finite polyhedral complex in $\R^m$ of
dimension $\dim(X)$, which is the closure of the image of $X(K)$ under the coordinatewise valuation
map. 

Assume that $X = C$ is of dimension $1$ and let $\Gamma = \trop(X)$. Then $\Gamma$ is a (non-compact) tropical curve in the sense of this article. The map 
\[
    \begin{tikzcd}[row sep=0pt,/tikz/column 1/.append style={anchor=base east}
     ,/tikz/column 2/.append style={anchor=base west}]
        \Div(C) \arrow[r] & \Div(\Gamma), \quad 
        \displaystyle\sum_ia_i p_i \arrow[r,mapsto] & \displaystyle\sum_ia_i\trop(p_i), 
    \end{tikzcd}
\]
is clearly degree-preserving. 

One can moreover define an \emph{extended tropicalization map} of subvarieties of toric varieties; see \cite[\S\ 3, 4]{payne-ana}. Indeed, let $X \subseteq Y_{\Delta}$ be a subvariety of a toric variety $Y_{\Delta}$, defined over $K$, corresponding to a fan $\Delta$ in $N_{\R}$. In \emph{loc.\;cit.\;}S.~Payne constructs an extended tropicalization map 
\[
    \begin{tikzcd}
        \overline{\trop} \colon X(K) \ar[r] & N_{\Delta},
    \end{tikzcd}
\]
where $N_{\Delta} = \bigsqcup_{\sigma \in \Delta}N_{\R}/\langle \sigma \rangle$
is the Kajiwara–Payne compactified tropical toric variety, a partial compactification of $N_{\R}$
 obtained by adding strata corresponding to cones of the fan. The \emph{extended tropicalization of} $X$ is defined as the closure of the image $\overline{\trop}(X)$ inside $N_{\Delta}$. 

Assume $X = C \subset Y_{\Delta}$ is of dimension one and let $\Gamma$ be the extended tropicalization of $C$. The boundary $\partial\Gamma = \Gamma \smallsetminus \Gamma^{\circ}$ consists of a finite set of points, called the \emph{infinite vertices}. If all infinite vertices have valence one, or equivalently, if no two unbounded rays meet at infinity, then $\Gamma$ is naturally endowed with a metric multigraph structure and hence can be identified with a compact tropical curve; see \cite[\S\ 2]{jell}.

In this case, the map 
 \[
    \begin{tikzcd}[row sep=0pt,/tikz/column 1/.append style={anchor=base east}
     ,/tikz/column 2/.append style={anchor=base west}]
        \Div(C) \arrow[r] & \Div(\Gamma), \quad 
        \displaystyle\sum_ia_i p_i \arrow[r,mapsto] & \displaystyle\sum_ia_i\overline{\trop}(p_i),
    \end{tikzcd}
\]
is again clearly degree preserving. Note that if $D$ is a divisor on $C$ which does not meet the toric boundary, then $\overline{\trop}(D)$ is supported on $\Gamma^{\circ}$. Arguing as in the proof of \autoref{prop:specialization} we conclude that 
\begin{equation}\label{eq:vol-tro}
    \vol_C(D) = \max\!\left\{\deg\!\big(\overline{\trop}(D)\big), 0\right\} = \vol_{\Gamma}^{\mathbb{T}}\!\left(\overline{\trop}(D)\right)
\end{equation}
for any $D \in \Div(C)$. \

\begin{rem}
    If $K$ is discretely valued, then the (extended) tropicalization map factors through the specialization map, as defined in the previous Subsection. More generally, if $K$ is not necessarily discretely valued, then the analog of the specialization map is the Berkovich retraction map $\tau\colon C^{\operatorname{an}} \to \Gamma$ from the Berkovich analitification $C^{\operatorname{an}}$ of $C$ to an (extended) skeleton $\Gamma$, associated to any strongly semistable model $\mathfrak{C}$ of $C$ over the valuation ring of $K$. The (extended) tropicalization map extends continuously to $\overline{\trop} \colon C^{\operatorname{an}} \to N_{\Delta}$
    and there exists an integral, piecewise affine map $\gamma \colon \Gamma \to N_{\Delta}$ such that $\overline{\trop} = \gamma \circ \tau$; see \cite[Prop.\ 5.4 (1)]{BPR} for the case that $C$ is contained in a torus, and \cite[\S\ 2.4]{jell} for the extended version.
\end{rem}

\begin{example}\label{exa: trop via valuation}
    Let $K=\BC\{\!\{t\}\!\}$ be the field of \emph{Puiseux series} endowed with the $t$-adic valuation $\nu:K \to \BT$, where $\nu(0)=\infty$. 
    The valuation $\nu$ is naturally extended to $K^n$ by taking coordinatewise valuation, and to the ring of Laurent polynomials $K[x^{\pm},y^{\pm},z^{\pm}]$ by 
    \[
        \nu\Big(\sum a_{\alpha,\beta,\gamma}x^\alpha y^\beta z^\gamma\Big)\coloneqq \min\{\nu(a_{\alpha,\beta,\gamma}) +\alpha x + \beta y + \gamma z\}.
    \]
    
    Let $F\coloneqq t(x^3+y^3+z^3)+xyz\in K[x,y,z]$, and consider $X = V(F)$ the corresponding hypersurafce, which can be thought of as a degenerating family of elliptic curves, with degenerate fiber over $t=0$.  It is naturally embedded in the projective plane $\BP^2_K$. 
    
    Let $\widetilde{X} \coloneqq X \cap \BG_{K}^3$, and passing to the affine chart $\BP_K^2\smallsetminus V(z)$, we consider $\widetilde{X}=V(\widetilde{F})\subseteq \BG_{K}^2$ where $\widetilde{F}=t(r^3 + s^3 + 1) + rs\in K[r^{\pm},s^{\pm}]$.
    The tropicalization $\widetilde{F}_{\trop}$ of $\widetilde{F}$ is given by 
    \[
        \nu(\widetilde{F})=\min\{1 + 3r,\;1 + 3s,\;1,\; r + s\}.
    \]
    Its tropical corner locus $V_{\trop}(\widetilde{F}_{\trop})$ is the set of points in $\R^2$ where $\widetilde{F}_{\trop}$ attains the minimum at least twice. 
    As a special case of the Fundamental Theorem of Tropical Geometry (FTTG), it follows that the tropicalization $\trop(\widetilde{X})$ of $\widetilde{X}$ is given by
    \[
        V_{\trop}(\widetilde{F}_{\trop})=\overline{\nu\big(V(\widetilde{F})\big)} \subset \BR^2;
    \]
    see \cite[Chapter 3]{Maclagan_Sturmfels_2015} for more details concerning tropicalization and the FTTG, and \autoref{fig: tropicalization of the elliptic curve} to visualize the tropical curve $\trop(\widetilde{X})$.
    \begin{figure}[ht]
    \centering
    \begin{tikzpicture}[scale=1, line cap=round, line join=round, every node/.style={font=\normalfont}]
        \def\a{2}
        \draw[->, line width=.8pt] (-\a-.3,0) -- (\a+.5,0) node[below] {$x$};
        \draw[->, line width=.8pt] (0,-\a-.2) -- (0,\a+.5) node[left]  {$y$};

        \coordinate (A) at (1,0);
        \coordinate (B) at (0,1);
        \coordinate (C) at (-1,-1);
        \coordinate (D) at (\a/4,\a/4);

        \draw[red, line width=1.3pt] (A) -- (B) -- (C) -- cycle;
        \draw[red, line width=1.3pt] (A) -- (\a,0);
        \draw[red, line width=1.3pt] (B) -- (0,\a);
        \draw[red, line width=1.3pt] (C) -- (-\a,-\a);

        \fill (A) circle (1.6pt) node[below right] {$(1,0)$};
        \fill (B) circle (1.6pt) node[above right]  {$(0,1)$};
        \fill (C) circle (1.6pt) node[above left]  {$(-1,-1)$};
        \fill (D) circle (1.6pt) node[right]  {$w$};
    \end{tikzpicture}
    \begin{tikzpicture}[scale=1, line cap=round, line join=round, every node/.style={font=\normalfont}]
        \def\a{2}
        \draw[->, line width=.8pt] (-\a-.3,0) -- (\a+.5,0) node[below] {$x$};
        \draw[->, line width=.8pt] (0,-\a-.2) -- (0,\a+.5) node[left]  {$y$};

        \coordinate (A) at (1,0);
        \coordinate (B) at (0,1);
        \coordinate (C) at (-1,-1);
        \coordinate (D) at (\a/4,\a/4);
        \coordinate (E) at (\a,0);
        \coordinate (F) at (0,\a);
        \coordinate (G) at (-\a,-\a);
        \node[] at (-5,0) {};

        \draw[red, line width=1.3pt] (A) -- (B) -- (C) -- cycle;
        \draw[red, line width=1.3pt] (A) -- (\a,0);
        \draw[red, line width=1.3pt] (B) -- (0,\a);
        \draw[red, line width=1.3pt] (C) -- (-\a,-\a);

        \fill (A) circle (1.6pt) node[below right] {$(1,0)$};
        \fill (B) circle (1.6pt) node[above right]  {$(0,1)$};
        \fill (C) circle (1.6pt) node[above left]  {$(-1,-1)$};
        \fill (D) circle (1.6pt) node[right]  {$w$};
        \fill (E) circle (1.6pt) node[right]  {};
        \fill (F) circle (1.6pt) node[right]  {};
        \fill (G) circle (1.6pt) node[right]  {};
    \end{tikzpicture}
    \caption{The tropicalization $\trop(\widetilde{X})$ of the elliptic curve $\widetilde{X}$ on the left, and its compactification on the right.}
    \label{fig: tropicalization of the elliptic curve}
\end{figure}
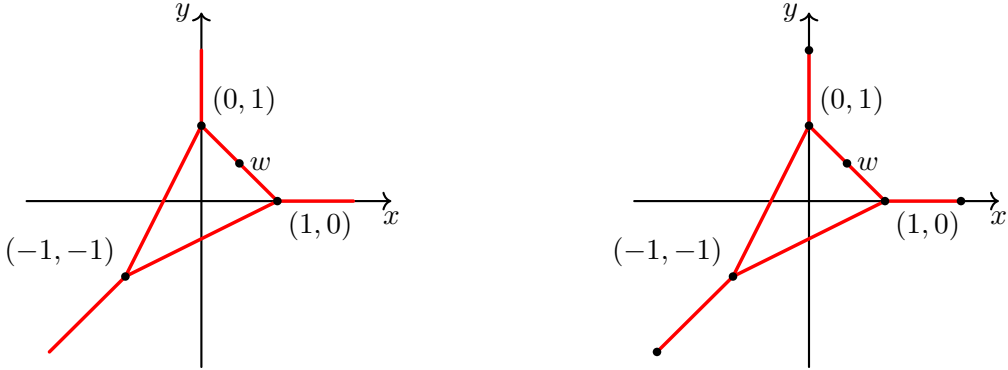

Note that we don't have a well-defined map 
\[
    \trop \colon \Div(X) \longrightarrow \Div\!\big(\!\trop(\widetilde{X})\big), \quad t\sum a_{P}P \longmapsto \sum a_{P}\trop(P),
\]
since $P$ could be a toric boundary point.
For example, consider the divisor $D$ cut out by the line $y + x = 0$, that is, $D= P_{+} + P_{-} + P_{\infty}$, where $P_{+}=[t^{1/2}:-t^{1/2}:1]$, $P_{-}=[-t^{1/2}:t^{1/2}:1]$ and $P_{\infty}=[1:-1:0]$. We have $W \coloneqq \nu(P_{+})=\nu(P_{-})=[1/2:1/2:0]$, and $P \coloneqq \nu(P_\infty)=[0:0:\infty]$. Hence, while $W$ corresponds to the point $w=(1/2,1/2)$ in $\trop(\widetilde{X})$, there is no point in $\trop(\widetilde{X})$ corresponding to $P$. 

It is possible to avoid this inconsistency by considering the extended tropicalization 
\[
    \begin{tikzcd}
        \overline{\trop} \colon X \ar[r] & N_{\Delta}, 
    \end{tikzcd}
\]
where $\Delta$ is the fan of the toric variety $\mathbb{P}^2_K$.
The extended tropicalization $\overline{\trop(X)}$ can be seen in \autoref{fig: tropicalization of the elliptic curve}, on the right. 

Now, note that the point $\nu(P_{\infty})$ corresponds to the endpoint of the ray adjacent to the vertex $(-1,-1)$. We may now consider the degree preserving map 
\[
    \begin{tikzcd}[row sep=0pt,/tikz/column 1/.append style={anchor=base east}
     ,/tikz/column 2/.append style={anchor=base west}]
        \overline{\trop} \colon \Div(X) \arrow[r]  & \Div\!\big(\overline{\trop(X)}\big), \quad \sum a_{P}P \arrow[r,mapsto]  & \sum a_{P}\overline{\trop}(P)
    \end{tikzcd}
\]
and compute volumes of divisors tropically.
\end{example}

\begin{example}
We note that we can turn \autoref{exa: trop via valuation} into the setting of the specialization map by considering an appropriate discrete valued field. Indeed, let $R\coloneqq\BC[\![t^{1/2}]\!]$, $K$ its fraction field, and $\BC$ its residual field. Let $\mathfrak{C}/R$ be the following family of Fermat elliptic curves
    \[
        V\!\big(t(x^3+y^3+z^3)+xyz\big)\subset \BP^2_R.
    \] 
    Note that the special fiber $\mathfrak{C}_\BC$ is isomorphic to $V(xyz) \subset \BP^2_{\BC}$, and the generic fiber $\mathfrak{C}_{K} \subset \BP^2_K$ is smooth, namely it is an elliptic curve. Let $D\in\Div(\mathfrak{C}_{K})$ be the divisor cut out by the line $y + x = 0$, that is, $D= P_{+} + P_{-} + P_{\infty}$, where $P_{+}=[t^{1/2}:-t^{1/2}:1]$, $P_{-}=[-t^{1/2}:t^{1/2}:1]$ and $P_{\infty}=[1:-1:0]$. 
    \begin{figure}[ht]
    \centering
    \begin{tikzpicture}[>=Latex, line cap=round, line join=round,
    x=1cm, y=1cm, scale=.65, line width=1pt, every node/.style={font=\normalsize}]
        \def\W{9.0}
        \def\H{3.5}
        \def\ybase{-1.8}
        \def\xzero{2.0}
        \def\xn{6.6}

        \def\kappa{0.75}

        \def\fiberZeroCurv{0.25}
        \def\fiberZeroV{0.38}
        \def\fiberZeroBottom{0.10}
        \def\fiberZeroHeight{2.90}

        \def\fiberTBottom{0.55}
        \def\fiberTHeight{2.20}

        \coordinate (TL) at (0,\H);
        \coordinate (TR) at (\W,\H);
        \coordinate (BR) at (\W,0);
        \coordinate (BL) at (0,0);

        \coordinate (cR1) at ({\kappa*0.55},{\kappa*(-0.9)});
        \coordinate (cR2) at ({\kappa*0.55},{\kappa*( 0.9)});
        \coordinate (cB1) at ({\kappa*(-2.2)},{\kappa*(-0.35)});
        \coordinate (cB2) at ({\kappa*( 2.2)},{\kappa*(-0.35)});

        \draw (TR) .. controls ($(TR)+(cB1)$) and ($(TL)+(cB2)$) .. (TL)
              .. controls ($(TL)+(cR1)$) and ($(BL)+(cR2)$) .. (BL)
              .. controls ($(BL)+(cB2)$) and ($(BR)+(cB1)$) .. (BR)
              .. controls ($(BR)+(cR2)$) and ($(TR)+(cR1)$) .. (TR) -- cycle;

        \node at (4.5,\H + .4)
        {$V\!\big(t(x^3+y^3+z^3)+xyz\big)\subset \BP^2_R$};

        \coordinate (Ztop) at (\xzero, {\fiberZeroBottom + \fiberZeroHeight});
        \coordinate (Zmid) at (\xzero, {\fiberZeroBottom + 0.5*\fiberZeroHeight});
        \coordinate (Zbot) at (\xzero, {\fiberZeroBottom});

        \draw (Ztop) .. controls ($(Ztop)+(-\fiberZeroCurv,-\fiberZeroV)$)
                      and ($(Zmid)+(-\fiberZeroCurv, \fiberZeroV)$) .. (Zmid)
                      .. controls ($(Zmid)+(\fiberZeroCurv,-\fiberZeroV)$)
                      and ($(Zbot)+(\fiberZeroCurv, \fiberZeroV)$) .. (Zbot);

        \node at (\xzero - .7, {\fiberZeroBottom + 1.4}) {$\mathfrak{C}_{K}$};

        \coordinate (Tcenter) at (\xn, {\fiberTBottom + 0.55*\fiberTHeight - .5});
        \coordinate (Ta) at ($(Tcenter)+(0,1.10)$);
        \coordinate (Tb) at ($(Tcenter)+(-1.05,-0.75)$);
        \coordinate (Tc) at ($(Tcenter)+( 1.05,-0.75)$);

        \draw ($(Ta) + (.2,.3)$) -- ($(Tb) +(-.2,-.3)$);
        \draw ($(Tb) + (-.3,0)$) -- ($(Tc) + (.3,0)$);
        \draw ($(Tc) + (.2,-.3)$) -- ($(Ta) + (-.2,.2)$);

        \fill ($(Ta) + (-.03,-.1)$) circle (3pt);
        \fill (Tb) circle (3pt);
        \fill (Tc) circle (3pt);

        \node at ($(Tcenter)+(1,.3)$) {$\mathfrak{C}_{k}$};

        \draw (-0.1,\ybase) -- (\W+.3,\ybase);

        \fill (\xzero,\ybase) circle (3pt);
        \fill (\xn,\ybase)    circle (3pt);

        \node[below=4pt] at (\xzero,\ybase) {$\langle 0\rangle$};
        \node[below=4pt] at (\xn,\ybase) {$\langle t^{1/2}\rangle$};

        \draw[dashed] (\xzero,{\fiberZeroBottom - .1}) -- (\xzero,\ybase);
        \draw[dashed] (\xn,{\fiberTBottom - .5}) -- (\xn,\ybase);

        \def\xSpec{-1.5}

        \coordinate (SpecPoint) at (\xSpec,\ybase);
        \node at (SpecPoint) {$\mathrm{Spec}\,R$};

        \node at (\xSpec,2.8) {$\mathbb{P}^{2}_{R}$};

        \draw[->] (\xSpec,2.2) -- (\xSpec,\ybase+0.5);
        \node[left] at (\xSpec,.7) {$\pi$};
    \end{tikzpicture}
    \begin{tikzpicture}[scale=1, line width=1pt, every node/.style={font=\normalsize}]
            \def\r{.8}
            \foreach \i/\ang in {1/90,2/210,3/330}{
            \coordinate (V\i) at (\ang:\r);
            \fill (V\i) circle (2pt);
            }
            \node[] at ($(V1) + (90:.3)$) {$v_{x}$};
            \node[] at ($(V2) + (180:.5)$) {$v_{y}$};
            \node[] at ($(V3) + (0:.5)$) {$v_{z}$};
            \fill (30:\r/2) circle (2pt);
            \node[right] at (30:\r/2) {$w$};
            \node[right] at (260:1.5) {$\Gamma$};
            \node[right] at (270:2.5) {};
            \node[right] at (180:4) {};

            \foreach \i/\j in {1/2,2/3,3/1}{
            \draw (V\i) -- (V\j);
            }
        \end{tikzpicture}
        \caption{The family $\mathfrak C/R$ with smooth generic fiber $\mathfrak C_K$, and special fiber
    $\mathfrak C_k\cong V(xyz)$, a union of three $\mathbb P_k^1$'s meeting pairwise (on the left), and the dual multigraph $\Gamma$ of the special fiber $\mathfrak{C}_{k}$ (on the right).}
    \label{fig: degeneration_triangle}
    \end{figure}
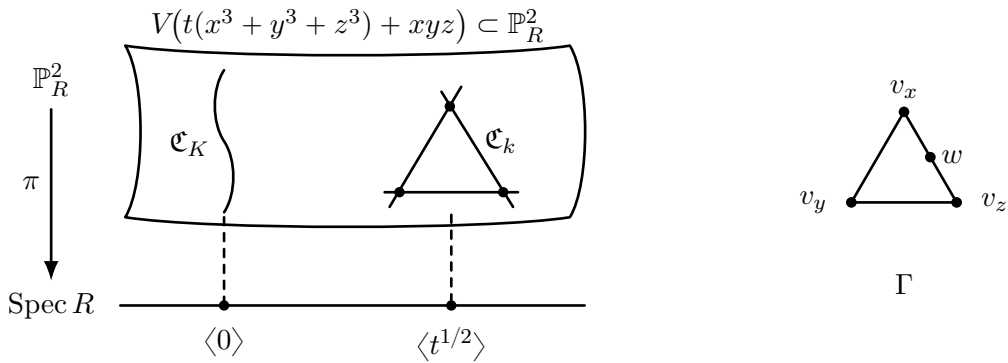

    Let $\rho:\Div(\mathfrak{C}_{K}) \to \Div(\Gamma)$ be the specialization map.  Note that $\rho(D_{K})=2w + v_z$; see \autoref{fig: degeneration_triangle}. Once more, as in \autoref{exa: vol(D)=vol(p(D))}, we can directly verify that $\deg(D_{K})=\vol_{\Gamma}^{\BT}\!\big(\rho(D_{K})\big)=3$.

\end{example}

\medskip

\bibliographystyle{amsalpha}
\bibliography{ref} 

\end{document}